\documentclass[a4paper, 10pt]{amsart}

\usepackage[all,cmtip]{xy}

\usepackage{fullpage}

\usepackage{tikz-cd}
\usetikzlibrary{arrows}
\usepackage{caption}
\usepackage{ stmaryrd }

\tikzset{
commutative diagrams/.cd,
arrow style=tikz,
diagrams={>=latex}}

\usepackage[utf8]{inputenc}
\usepackage{amsmath}
\usepackage{amsthm}
\usepackage{ mathrsfs }
\usepackage{ amssymb }
\usepackage{mathtools}
\usepackage{enumitem}

\usepackage{tikz}

\theoremstyle{definition}
\newtheorem{definition}{Definition}[section]
\newtheorem{proposition}{Proposition}[section]
\newtheorem{corollary}{Corollary}[section]

\newtheorem{lemma}{Lemma}[section]
\newtheorem{remark}{Remark}[section]
\newtheorem{theorem}{Theorem}[section]
\newtheorem*{theorem*}{Theorem}

\newcommand{\icg}{\mathsf{ICG}}
\newcommand{\mf}{\mathcal{F}}

\newcommand{\tr}{\mathfrak{tr}}
\newcommand{\sder}{\mathfrak{sder}}

\newcommand{\grt}{\mathfrak{grt}}
\newcommand{\frakt}{\mathfrak{t}}

\newcommand{\id}{\text{Id}}
\newcommand{\im}{\text{im}}

\newcommand{\graphs}{\mathsf{graphs}}
\newcommand{\GC}{\mathsf{GC}_2}
\newcommand{\Gra}{\mathsf{Gra}_2}

\newcommand{\lie}{\mathfrak{lie}}

\newcommand{\ad}{\text{ad}}
\newcommand{\depth}{\text{depth}}
\newcommand{\wt}{\text{wt}}

\newcommand{\poly}{\mathbb{K}[x,y,z]}
\newcommand{\polyodd}{\mathbb{K}[x,y,z]^{\text{odd}}}
\newcommand{\polyeven}{\mathbb{K}[x,y,z]^{\text{even}}}

\newcommand{\val}{\text{val}}

\title{FIltrations on Graph Complexes and the Grothendieck-Teichm\"uller Lie algebra in depth two}
\author{Matteo Felder}
\begin{document}

\begin{abstract}
We establish an isomorphism between the Grothendieck-Teichm\"uller Lie algebra $\grt_1$ in depth two modulo higher depth and the cohomology of the two-loop part of the graph complex of internally connected graphs $\icg(1)$. In particular, we recover all linear relations satisfied by the brackets of the conjectural generators $\sigma_{2k+1}$ modulo depth three by considering relations among two-loop graphs.

The Grothendieck-Teichm\"uller Lie algebra is related to the zeroth cohomology of M. Kontsevich's graph complex $\GC$ via T. Willwacher's isomorphism. We define a descending filtration on $H^0(\GC)$ and show that the degree two components of the corresponding associated graded vector spaces are isomorphic under T. Willwacher's map.
\end{abstract}

\address{%
Matteo Felder\\
Dept. of Mathematics\\
University of Geneva\\
2-4 rue du Li\`evre\\
1211 Geneva 4\\
Switzerland\\            
Matteo.Felder@unige.ch 
}

\maketitle
\tableofcontents

\section*{Introduction}

The Grothendieck-Teichm\"uller Lie algebra $\grt_1$ was introduced by V. Drinfeld \cite{Drinfeld1991}. Although being widely studied it still remains a somewhat mysterious object. It is spanned by series of Lie words in two variables $x$ and $y$ satisfying an antisymmetry, a hexagon and a pentagon equation. Its Lie bracket is given by,
\begin{equation*}
\{\psi_1,\psi_2\}:=D_{\psi_1}\psi_2-D_{\psi_2}\psi_1+[\psi_1,\psi_2],
\end{equation*}
where $[\psi_1,\psi_2]$ denotes the bracket on the completed free Lie algebra in $x$ and $y$, and $D_{\psi}$ is the unique derivation extending $x\mapsto 0$, $y\mapsto [y,\psi(x,y)]$. Moreover, $\grt_1$ is graded by weight (i.e. the total number of $x$ and $y$ of any Lie word) and filtered by depth (i.e. the minimal number of $y$ in any Lie word appearing in a series). Also, it is well-known that $\grt_1$ contains, for $k\geq 1$, elements $\sigma_{2k+1}$ of odd weight whose coefficient in front of the Lie word $\ad^{2k}_x(y)$ is non-vanishing (in particular, it can be normalized to be equal to one). An important result by F. Brown states that these elements freely generate a Lie subalgebra of $\grt_1$ \cite{Brown2012}. Conjecturally however, the free Lie algebra generated by the $\sigma_{2k+1}$ should coincide with $\grt_1$. 

Any Lie monomial appearing in the bracket of two such conjectural generators $\{\sigma_{2i+1},\sigma_{2j+1}\}$ contains at least two $y$. Modulo elements of depth three and higher, these brackets satisfy a set of linear relations of homogeneous weight. For instance in weight 12,
\begin{equation*}
\{\sigma_3,\sigma_9\}=3\{\sigma_5,\sigma_7\} \mod \mf^3\grt_1.
\end{equation*}
Linear relations of this sort where first studied by Y. Ihara and N. Takao \cite{Ihara2002}, and independently by A. Goncharov \cite{Goncharov2001}. In addition, Y. Ihara and N. Takao showed that modulo depth three, the brackets $\{\sigma_{2i+1},\sigma_{2j+1}\}$ actually generate the whole quotient,
\begin{equation*}
\mf^2\grt_1/\mf^3\grt_1.
\end{equation*}
Later, L. Schneps fully classified such linear identities by relating them to restricted even period polynomials associated to cusp forms on $\text{SL}_2(\mathbb{Z})$ \cite{Schneps2006}. The aim of this text is twofold. First, we add one more interpretation of the space spanned by the brackets $\{\sigma_{2i+1},\sigma_{2j+1}\}$ modulo $\mf^3\grt_1$, namely via the cohomology of a certain graph complex.

For this, consider the complex $(C,d_0)$,
\begin{equation*}
0\rightarrow C_0\overset{d_0}{\rightarrow} C_1\overset{d_0}{\rightarrow} C_2\rightarrow 0
\end{equation*}
where the graphs in $C_i$ are as depicted in Figure \ref{figure:Ccomplexintro}. The differential $d_0$ sums over the three possible ways of splitting the 4-valent vertices, thus attaching an additional ``hair" to each strand making up the theta shape of the graph. The cohomology of this complex was computed in \cite{Conant2005}, with the aim of determining the cohomology of the two-loop part of a larger complex introduced by G. Arone and V. Turchin to study the rational homotopy type of long embeddings \cite{AroneTurchin2015}. In degree one the cohomology can be identified with a quotient of the algebra of even polynomials in three variables. Moreover, we find that graphs $\theta_{2i,2j}$ for which only two of the main strand are decorated by $2i$ and $2j$ ``hair", respectively, form a generating set for $H^1(C,d_0)$. Our first main result then states the following.
\begin{theorem*}
The map
\begin{align*}
\mf^2\grt_1/\mf^3\grt_1&\rightarrow H^1(C,d_0)\\
\{\sigma_{2i+1},\sigma_{2j+1}\} \mod \mf^3\grt_1&\mapsto \theta_{2i,2j}
\end{align*}
is an isomorphism of vector spaces.
\end{theorem*}

In particular, there is a canonical equivalence between the linear relations satisfied by the brackets $\{\sigma_{2i+1},\sigma_{2j+1}\}$ modulo $\mf^3\grt_1$ and the cohomology classes represented by the graphs $\theta_{2i,2j}$. The proof of this theorem follows from one of the classification results of L. Schneps \cite{Schneps2006}. 

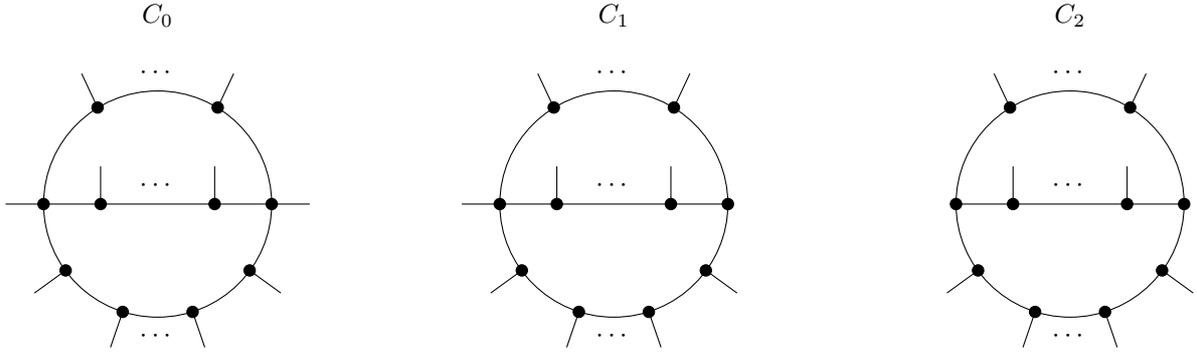
\begin{figure}[ht]
\centering
\begin{tikzpicture}

\node at (0,2.5) {$C_0$};
\node at (6,2.5) {$C_1$};
\node at (12,2.5) {$C_2$};

\node at (0,1.75) {$\cdots$};
\node at (0,0.25) {$\cdots$};
\node at (0,-1.75) {$\cdots$};

\draw (0,0) circle (1.5cm);
\draw(-1.5,0) -- (1.5,0);

\draw[black,fill=black](0.79,1.28) circle (0.075cm);
\draw[black,fill=black](-0.79,1.28) circle (0.075cm);

\draw[black,fill=black](0.75,0) circle (0.075cm);
\draw[black,fill=black](-0.75,0) circle (0.075cm);

\draw[black,fill=black](1.21,-0.88) circle (0.075cm);
\draw[black,fill=black](-1.21,-0.88) circle (0.075cm);
\draw[black,fill=black](0.46,-1.43) circle (0.075cm);
\draw[black,fill=black](-0.46,-1.43) circle (0.075cm);

\draw[black,fill=black](-1.5,0) circle (0.075cm);
\draw[black,fill=black](1.5,0) circle (0.075cm);

\draw(1.5,0) -- (2,0);
\draw(-2,0) -- (-1.5,0);

\draw(0.75,0) -- (0.75,0.5);
\draw(-0.75,0) -- (-0.75,0.5);

\draw(0.79,1.28) -- (1,1.73);
\draw(-0.79,1.28) -- (-1,1.73);

\draw(1.21,-0.88) -- (1.62,-1.18);
\draw(-1.21,-0.88) -- (-1.62,-1.18);
\draw(0.46,-1.43) -- (0.62,-1.9);
\draw(-0.62,-1.9) -- (-0.46,-1.43);


\node at (0+6,1.75) {$\cdots$};
\node at (6,0.25) {$\cdots$};
\node at (6,-1.75) {$\cdots$};

\draw (6,0) circle (1.5cm);
\draw(4.5,0) -- (7.5,0);

\draw[black,fill=black](6.79,1.28) circle (0.075cm);
\draw[black,fill=black](5.21,1.28) circle (0.075cm);

\draw[black,fill=black](6.75,0) circle (0.075cm);
\draw[black,fill=black](-0.75+6,0) circle (0.075cm);

\draw[black,fill=black](7.21,-0.88) circle (0.075cm);
\draw[black,fill=black](4.79,-0.88) circle (0.075cm);
\draw[black,fill=black](6.46,-1.43) circle (0.075cm);
\draw[black,fill=black](5.54,-1.43) circle (0.075cm);

\draw[black,fill=black](4.5,0) circle (0.075cm);
\draw[black,fill=black](7.5,0) circle (0.075cm);

\draw(4.5,0) -- (4,0);

\draw(6.79,1.28) -- (7,1.73);
\draw(5.21,1.28) -- (5,1.73);

\draw(0.75+6,0) -- (0.75+6,0.5);
\draw(-0.75+6,0) -- (-0.75+6,0.5);

\draw(7.21,-0.88) -- (7.62,-1.18);
\draw(4.79,-0.88) -- (4.38,-1.18);
\draw(6.46,-1.43) -- (6.62,-1.9);
\draw(5.38,-1.9) -- (5.54,-1.43);


\node at (0+12,1.75) {$\cdots$};
\node at (12,0.25) {$\cdots$};
\node at (12,-1.75) {$\cdots$};

\draw (12,0) circle (1.5cm);
\draw(4.5+6,0) -- (7.5+6,0);

\draw[black,fill=black](6.79+6,1.28) circle (0.075cm);
\draw[black,fill=black](5.21+6,1.28) circle (0.075cm);

\draw[black,fill=black](0.75+12,0) circle (0.075cm);
\draw[black,fill=black](-0.75+12,0) circle (0.075cm);

\draw[black,fill=black](7.21+6,-0.88) circle (0.075cm);
\draw[black,fill=black](4.79+6,-0.88) circle (0.075cm);
\draw[black,fill=black](6.46+6,-1.43) circle (0.075cm);
\draw[black,fill=black](5.54+6,-1.43) circle (0.075cm);

\draw[black,fill=black](4.5+6,0) circle (0.075cm);
\draw[black,fill=black](7.5+6,0) circle (0.075cm);

\draw(6.79+6,1.28) -- (7+6,1.73);
\draw(5.21+6,1.28) -- (11,1.73);

\draw(0.75+12,0) -- (0.75+12,0.5);
\draw(-0.75+12,0) -- (-0.75+12,0.5);

\draw(7.21+6,-0.88) -- (7.62+6,-1.18);
\draw(4.79+6,-0.88) -- (4.38+6,-1.18);
\draw(6.46+6,-1.43) -- (6.62+6,-1.9);
\draw(5.38+6,-1.9) -- (5.54+6,-1.43);

\end{tikzpicture}
\caption{Graphs in $C_0$, $C_1$, $C_2$.}\label{figure:Ccomplexintro}
\end{figure}

In a second part, we show that the equivalence between the cohomology of two-loop graphs and the Grothendieck-Teichm\"uller Lie algebra in depth two modulo higher depth is no coincidence. For instance, $\grt_1$ is related to graphs via T. Willwacher's important result \cite{Willwacher2014} which states that there is an isomorphism of Lie algebras,
\begin{equation*}
H^0(\GC)\cong \grt_1,
\end{equation*}
where $\GC$ denotes a variant of Kontsevich's graph complex. Our aim is to define a filtration on $H^0(\GC)$ which is compatible with the depth filtration on $\grt_1$ under T. Willwacher's isomorphism. Our second main result states that this is can be done in depth two modulo higher depth.
\begin{theorem*}
T. Willwacher's map induces an isomorphism,
\begin{equation*}
\mf^2H^0(\GC)/\mf^3H^0(\GC) \overset{\cong}{\longrightarrow}\mf^2\grt_1/\mf^3\grt_1.
\end{equation*}
\end{theorem*}

In the proof of this theorem we show that on the degree two part of the associated graded, the map induced by T. Willwacher's isomorphism factors through an isomorphism,
\begin{equation*}
\mf^2H^0(\GC)/\mf^3H^0(\GC)\overset{\cong}{\longrightarrow} H^1(C,d_0).
\end{equation*}
Together with our first main result, this implies the statement. 

The hope at this point is that the reasoning above can be generalized. More precisely, we expect that the filtration on $H^0(\GC)$ is defined in such a way that we have an isomorphism on the entire associated graded. Additionally, the degree $p$ part of the associated graded could then ideally be identified with the cohomology of the $p$-loop part of a graph complex, similar to to the one of which $C$ is a subcomplex, i.e
\begin{equation*}
\mf^p H^0(\GC)/\mf^{p+1} H^0(\GC)\cong\text{ ``cohomology of } p \text{-loop graphs"}.
\end{equation*}
This would in turn allow us to study the Grothendieck-Teichm\"uller Lie algebra's relations in higher depth by computing the cohomology of the space of higher loop order graphs. Possibly, such a cohomology could again be expressed in terms of some quotient of an algebra of polynomials, thus simplifying some of the relations found in the associated graded with respect to the depth filtration on $\grt_1$. For now, however, these questions remain open.

\section*{Acknowledgements} 
I am very grateful to my advisor Thomas Willwacher for his guidance and for numerous useful discussions and suggestions. I thank Anton Alekseev for his interest in the problem and his constant support. Discussions with Ricardo Campos, Florian Naef and Elise Raphael have also been extremely helpful, and I thank them for their encouragement. This work was supported by the grant MODFLAT of the European Research Council (ERC).

\section{The Grothendieck-Teichm\"uller Lie algebra}

\subsection{Definitions}
We recall the definition of the Grothendieck-Teichm\"uller Lie algebra $\grt_1$. It was first introduced by V. Drinfeld \cite{Drinfeld1991}.
\begin{definition}
The \emph{Drinfeld-Khono Lie algebra} $\frakt_n$ is generated by elements $t^{i,j}=t^{j,i}$, where $1\leq i,j \leq n$, which satisfy the relations,
\begin{align*}
[t^{i,j},t^{k,l}]=&0 \text{ if } \#\{i,j,k,l\}=4,\\
[t^{i,j}+t^{i,k},t^{j,k}]=&0 \text{ if } \#\{i,j,k\}=3.
\end{align*}
\end{definition}
\begin{definition}
The \emph{Grothendieck-Teichm\"uller Lie algebra} $\grt_1$ is spanned by elements $\psi\in \widehat{\lie_2}$ (the completed free Lie algebra in two generators $x,y$) satisfying the following relations:
\begin{align*}
\psi(x,y)&=-\psi(y,x)\\
\psi(x,y)+\psi(y,z)+\psi(z,x)&=0 \text{ for } x+y+z=0\\
\psi(t^{1,2},t^{2,3}+t^{2,4})+\psi(t^{1,3}+t^{2,3},t^{3,4})&=\psi(t^{2,3},t^{3,4})+\psi(t^{1,2}+t^{1,3},t^{2,4}+t^{3,4})+\psi(t^{1,2},t^{2,3})
\end{align*}
where the last equation takes values in the Lie algebra $\mathfrak{t}_4$. The bracket on $\grt_1$ is given by
\begin{equation*}
\{\psi_1,\psi_2\}:=D_{\psi_1}\psi_2-D_{\psi_2}\psi_1+[\psi_1,\psi_2],
\end{equation*}
where $D_{\psi}$ is the unique derivation extending $x\mapsto 0$, $y\mapsto [y,\psi(x,y)]$.
\end{definition}
\begin{remark}
It is well-known that $\grt_1$ contains elements $\sigma_{2k+1}$, for $k\geq 1$, for which the coefficient in front of the Lie word $\ad_x^{2k}(y)$ is non-vanishing. As it is a common convention, we shall assume this coefficient to be $1$. It has been shown by F. Brown that the Lie algebra freely generated by these $\sigma_{2k+1}$ forms a Lie subalgebra of $\grt_1$ \cite{Brown2012}. In fact, it is conjectured that these elements freely generate $\grt_1$ as a Lie algebra.
\end{remark}

\begin{definition}
Let $a(x,y)\in \lie_2$. The \emph{depth} $\depth(a)$ of $a$ is the minimal number of $y$'s contained in any of its Lie monomials. Let $b$ be a Lie word in $x,y$. The \emph{weight} $\wt(b)$ is the total number of $x$ and $y$ in $b$.
\end{definition}
Define the following descending filtration on $\grt_1$ by
\begin{equation*}
\mf^p\grt_1:=\{\psi\in \grt_1| \depth(\psi)\geq p\}
\end{equation*}
for $p\geq 1$. Note that $\mf^1\grt_1=\grt_1$. It is not hard to see that the filtration is compatible with the Lie algebra structure, i.e. for all $p,q\geq 1$,
\begin{align*}
\{\mf^p\grt_1,\mf^q\grt_1\}\subset& \mf^{p+q}\grt_1.
\end{align*}

\begin{remark}
Note that the weight defines a grading on $\grt_1$. For $k\geq 2$, we denote by $[\grt_1]_k$ the subspace of $\grt_1$ spanned by homogeneous elements of weight $k$.
\end{remark}

\subsection{Relations among the conjectural generators}

Next, we recall some useful results on the brackets of the conjectural generators. Consider the composition, 
\begin{equation*}
\bigoplus\limits_{i,j\geq 1} \mathbb{K}\cdot\{\sigma_{2i+1},\sigma_{2j+1}\}\overset{i}{\hookrightarrow}\mf^2\grt_1\overset{\pi}{\twoheadrightarrow}\mf^2\grt_1/\mf^3\grt_1.
\end{equation*}
Y. Ihara and N. Takao proved the following.
\begin{theorem}[\cite{Ihara2002}, Corollary to Theorem III-II-1.]\label{Thm:Ihara}
In depth two modulo depth three the brackets of the conjectural generators span the whole of $\grt_1$. That is, the composition,
\begin{align*}
\pi\circ i: \bigoplus\limits_{i,j\geq 1}\mathbb{K} \cdot\{\sigma_{2i+1},\sigma_{2j+1}\}&\rightarrow \mf^2\grt_1/\mf^3\grt_1\\
\{\sigma_{2i+1},\sigma_{2j+1}\}&\mapsto \{\sigma_{2i+1},\sigma_{2j+1}\}\mod \mf^3\grt_1
\end{align*}
is surjective.
\end{theorem}

\begin{remark}\label{rmk:relations}
A. Goncharov \cite{Goncharov2001}, and independently Y. Ihara and N. Takao \cite{Ihara2002} discovered that for $k\geq 4$ even, there are $\lfloor \frac{k-4}{4}\rfloor-\lfloor\frac{k-2}{6}\rfloor$ linear relations modulo $\mf^3\grt_1$ among the brackets $\{\sigma_{2i+1},\sigma_{k-1-2i} \}$, $1\leq i \leq \frac{k-4}{2}$. That is, relations of the form,
\begin{equation*}
\sum\limits_{i=1}^{\frac{k-4}{2}}a_i \{\sigma_{2i+1},\sigma_{k-1-2i}\}=0\mod \mf^3\grt_1.
\end{equation*}
In weight $k$ and depth 2 modulo higher depth, and taking symmetry into account, we find that there are $\lfloor \frac{k-4}{4} \rfloor$ different brackets. The result above now implies that the dimension of the vector space spanned by the brackets of weight $k$ and depth $2$ modulo elements of higher depth is $\lfloor \frac{k-2}{6}\rfloor$. Together with Theorem \ref{Thm:Ihara}, we obtain that the dimension of the quotient space $[\mf^2\grt_1/\mf^3\grt_1]_k$ of $\grt_1$-elements of weight $k$ and depth $2$ modulo higher depth is also of $\lfloor \frac{k-2}{6}\rfloor$.
\end{remark}

The linear relations mentioned above were later fully classified by L. Schneps in \cite{Schneps2006} and we recall her result below. Consider the ring $\mathbb{K}\langle x,y \rangle$ of polynomials in two non-commutative variables $x$ and $y$. Denote by $\mathbb{K}^d\langle x,y \rangle$ the vector space of $\mathbb{K}\langle x,y \rangle$ generated by monomials containing exactly $d$ $y$'s. There is a canonical $S_{d+1}$-action on monomials in $\mathbb{K}^d\langle x,y \rangle$ given by,
\begin{equation*}
\sigma.(x^{i_1}yx^{i_2}y\dots x^{i_d}yx^{i_{d+1}})=x^{i_{\sigma(1)}}yx^{i_{\sigma(2)}}y\cdots x^{i_{\sigma(d)}}yx^{i_{\sigma(d+1)}},
\end{equation*}
for $\sigma\in S_{d+1}$. With this notation, L. Schneps' theorem states the following.

\begin{theorem}[\cite{Schneps2006}, Theorem 4.1. and Corollary 4.2.]\label{thm:Schneps}
Let 
\begin{equation*}
F=\sum\limits_{i=1}^{\frac{k-4}{2}} a_i \ad_x^{2i}(y) \ad_x^{k-2-2i}(y)\in \mathbb{K}^2\langle x,y \rangle
\end{equation*}
where $\ad_x(y)=xy-yx$. Then $F$ satisfies 
\begin{equation*}
F+(13).F=0 \text{ and } F+(123).F+(132).F=0
\end{equation*}
if and only if
\begin{equation}\label{eq:relation}
\sum\limits_{i=1}^{\frac{k-4}{2}}a_i \{\sigma_{2i+1},\sigma_{k-1-2i}\}=0\mod \mf^3\grt_1.
\end{equation}
\end{theorem}

\begin{remark}\label{rmk:identification}
It will be useful to note the following easy technical subtlety. The polynomial expansion of $F$ as in the theorem above is given by,
\begin{equation*}
F=\sum\limits_{i=1}^{\frac{k-4}{2}} a_i \sum\limits_{u=0}^{2i}\sum\limits_{v=0}^{k-2-2i}(-1)^{u+v}{2i \choose u}{k-2-2i\choose v}x^u yx^{2i-u+v}yx^{k-2-2i-v}.
\end{equation*}
Since we are mainly interested in the $S_3$-action, the relevant information is contained in the exponents of $x$. These may be encoded in the expression,
\begin{equation*}
G=\sum\limits_{i=1}^{\frac{k-4}{2}} a_i \sum\limits_{u=0}^{2i}\sum\limits_{v=0}^{k-2-2i}(-1)^{u+v}{2i \choose u}{k-2-2i\choose v}\alpha^u \beta^{2i-u+v} \gamma^{k-2-2i-v}=\sum\limits_{i=1}^{\frac{k-4}{2}} a_i (\alpha-\beta)^{2i}(\beta-\gamma)^{k-2-2i}.
\end{equation*}
The $S_3$-action on $\mathbb{K}^2\langle x,y \rangle$ is compatible with the $S_3$-action on $\mathbb{K}[\alpha,\beta,\gamma]$ which for $\sigma\in S_3$ is given by,
\begin{equation*}
\sigma. (\alpha^{k_1}\beta^{k_2}\gamma^{k_3})=\alpha^{k_{\sigma(1)}}\beta^{k_{\sigma(2)}}\gamma^{k_{\sigma(3)}}
\end{equation*}
on monomials. Clearly, keeping this identification in mind, L. Schneps' result reads as follows.
\begin{equation*}
G=\sum\limits_{i=1}^{\frac{k-4}{2}} a_i (\alpha-\beta)^{2i}(\beta-\gamma)^{k-2-2i}
\end{equation*}
satisfies $G+(13).G=0$ and $G+(123).G+(132).G=0$ if and only if equation \eqref{eq:relation} holds.
\end{remark}

\section{Internally connected graphs}

\subsection{Definitions}
The graph complex of internally connected graphs was introduced by P. Severa and T. Willwacher  \cite{Severawillwacher2011} based on the works of M. Kontsevich (see for instance \cite{Kontsevich1999}, \cite{Lambrechts2014}). Fix $n\geq 1$.
\begin{definition}\label{def:icg}
An \emph{admissible graph} is an unoriented graph $\Gamma$ with labeled vertices $1,2,\dots, n$ (called external), possibly other vertices (unlabeled and called internal) satisfying the following properties:
\begin{enumerate}
\item{There is a linear order on the set of edges.}
\item{$\Gamma$ has no double edges, nor simple loops (edges connecting a vertex with itself).}
\item{Every internal vertex is at least trivalent.}
\item{Every internal vertex can be connected by a path with an external vertex.}
\end{enumerate}
\end{definition}
Let $\graphs(n)$ be the vector space spanned by finite linear combinations of admissible graphs with $n$ external vertices, modulo the relation $\Gamma^\sigma=(-1)^{|\sigma|} \Gamma$, where $\Gamma^{\sigma}$ differs from $\Gamma$ by a permutation $\sigma$ on the order of edges. Here $|\sigma|$ denotes the parity of the permutation $\sigma$.
\begin{definition}
A graph in $\graphs(n)$ which is connected after we delete all external vertices is called \emph{internally connected}. Denote by $\icg(n)$ the space spanned by internally connected graphs modulo sign relations obtained from the order of edges. Define the grading on $\icg(n)$ to be,
\begin{equation*}
\deg\Gamma=1-\#\text{edges}+2\#\text{internal vertices}.
\end{equation*}
Set the differential $d$ on $\icg(n)$ (on $\graphs(n)$, respectively) to be given by vertex splitting. More precisely, an external vertex splits into an external and an internal vertex connected by an edge, and we sum over all possible ways of reconnecting the ``loose'' edges to the two newly created vertices. Similarly, an internal vertex splits into two internal vertices, before summing over all ways of reconnecting the edges previously connected to the splitted vertex. In both cases, we only keep graphs that are still internally connected (admissible, respectively). As a convention, we set the newly created edge to come last in the new order of the edges. In this way, we have $d^2=0$.
\end{definition}

\begin{remark}
The collections $\{\graphs(n)\}_{n\geq 1}$ and $\{\icg(n)\}_{n \geq 1}$ both form non-symmetric operads in the category of cochain complexes. The operadic composition in $\graphs$ (and also in $\icg$) is given by insertion. That is, for $\Gamma_1\in \graphs(r)$, $\Gamma_2\in \graphs(s)$,
\begin{equation*}
\Gamma_1 \circ_j \Gamma_2\in \graphs(r+s-1)
\end{equation*}
is constructed by replacing the $j$th external vertex by $\Gamma_2$, summing over all possible ways of reconnecting the ``loose" edges (which were previously adjacent to vertex $j$) to vertices of $\Gamma_2$, and keeping only admissible graphs (in the case of $\icg$, we only keep the internally connected ones). The order on the set of edges of the new graphs is simply given by letting the edges of $\Gamma_1$ come before those of $\Gamma_2$ while leaving the respective ordering unchanged.
\end{remark}
\begin{remark}
Since any graph in $\graphs(n)$ may be written as the disjoint union of its internally connected components (after identifying the external vertices), the internally connected graphs freely generate $\graphs(n)$ as a coalgebra. For a suitable choice of grading on $\graphs(n)$, we therefore have an isomorphism of cocommutative coalgebras
\begin{equation*}
\graphs(n)\cong S(\icg(n)[1]).
\end{equation*}
By definition, the differential on $\graphs(n)$ defines an $L_\infty$-structure on the graded vector space $\icg(n)$.
\end{remark}

\begin{definition}
Let $\Gamma\in \icg(n)$. An \emph{internal loop} of $\Gamma$ is a loop in $\Gamma$ which does not pass through any external vertex.
\end{definition}
The space $\icg(n)$ is filtered by the number of internal loops. That is, we may define the following descending filtration on $\icg(n)$,
\begin{equation*}
\mf^p\icg(n):=\{\Gamma\in \icg(n)| \Gamma \text{ has at least } p \text{ internal loops}\}.
\end{equation*}
Consider the spectral sequence corresponding to this filtration. On its first page, we find the cohomology of the associated graded complex,
\begin{equation*}
E_1^{p,q}(n)=H^{p+q}(\mf^p\icg(n)/\mf^{p+1}\icg(n),d_0)
\end{equation*}
where $d_0$ is the part of the of the differential $d$ on $\icg(n)$ which does not create any internal loops (i.e. it splits internal vertices only).

\begin{remark}\label{rmk:page}
The first page of this spectral sequence turns out to be very useful. It is related to the works of A. Alekseev and C. Torossian on the Kashiwara-Vergne conjecture \cite{Alekseev2012}. For instance, 
\begin{equation*}
\bigoplus\limits_{q\in\mathbb{Z}} E_1^{0,q}(n)=E_1^{0,0}(n)=H^0(\icg(n)/\mf^1\icg(n),d_0)
\end{equation*}
consists of (internally) 3-valent trees modulo the IHX relation and forms a Lie algebra. It corresponds bijectively (as a Lie algebra) to the Lie algebra of special derivations $\sder_n$ (for an introduction, see \cite{Alekseev2012}). Moreover, 
\begin{equation*}
\bigoplus\limits_{q\in\mathbb{Z}} E_1^{1,q}(n)=E_1^{1,0}(n)=H^1(\mf^1\icg(n)/\mf^2\icg(n),d_0)
\end{equation*}
describes the space of one-loop graphs, again modulo IHX. These correspond to a quotient of the space of cyclic words $\tr_n$ in $n$ variables. These equivalences are made precise in P. \v Severa and T. Willwacher's paper \cite{Severawillwacher2011}. These two spaces together with the induced differential on the first page yield a definition of the Kashiwara-Vergne Lie algebra in terms of graphs (\cite{Alekseev2012},\cite{Severawillwacher2011}). Furthermore, using some additional simplicial structure on $\icg(n)$, we constructed a nested sequence of Lie subalgebras of $E_1^{0,0}(2)$ interpolating between the Kashiwara-Vergne Lie algebra and the Grothendieck-Teichm\"uller Lie algebra \cite{Felder2016}. In this text we describe an application of the two-loop part, that is,
\begin{equation*}
\bigoplus\limits_{q\in\mathbb{Z}} E_1^{2,q}(n)=\bigoplus\limits_{q\in\mathbb{Z}} H^{2+q}(\mf^2\icg(n)/\mf^3\icg(n),d_0).
\end{equation*}
\end{remark}

\begin{figure}[ht]
\centering
\begin{tikzpicture}
\draw[black,fill=black](-0.75,1.5) circle (0.075cm);
\draw[black,fill=black](0.75,1.5) circle (0.075cm);
\draw(-0.75,0) circle (0.1cm);
\draw(0.75,0) circle (0.1cm);

\draw(-0.75,0) -- (-0.75,1.5);
\draw(-0.75,0) -- (0.75,1.5);
\draw(0.75,0) -- (-0.75,1.5);
\draw(0.75,0) -- (0.75,1.5);
\draw(-0.75,1.5) -- (0.75,1.5);

\node at (-0.75,-0.5) {$1$};
\node at (0.75,-0.5) {$2$};


\draw[black,fill=black](2.75,1.75) circle (0.075cm);
\draw[black,fill=black](4.25,1.75) circle (0.075cm);
\draw[black,fill=black](3.5,2.3) circle (0.075cm);
\draw[black,fill=black](3.5,1.2) circle (0.075cm);
\draw(2.5,0) circle (0.1cm);
\draw(3.5,0) circle (0.1cm);
\draw(4.5,0) circle (0.1cm);

\draw(2.5,0) -- (2.75,1.75);
\draw(2.5,0) -- (3.5,2.3);
\draw(3.5,0) -- (3.5,1.2);
\draw(4.5,0) -- (4.25,1.75);
\draw(3.5,2.3) -- (2.75,1.75);
\draw(2.75,1.75) -- (3.5,1.2);
\draw(3.5,1.2) -- (4.25,1.75);
\draw(3.5,2.3) -- (4.25,1.75);

\node at (2.5,-0.5) {$1$};
\node at (3.5,-0.5) {$2$};
\node at (4.5,-0.5) {$3$};


\draw[black,fill=black](6.5,1.4) circle (0.075cm);
\draw[black,fill=black](8.5,1.75) circle (0.075cm);
\draw[black,fill=black](6.5,2.1) circle (0.075cm);
\draw[black,fill=black](7.5,2.5) circle (0.075cm);
\draw[black,fill=black](7.5,1) circle (0.075cm);
\draw(7.5,0) circle (0.1cm);

\draw(7.5,0) -- (8.5,1.75);
\draw(7.5,0) -- (7.5,2.5);
\draw(7.5,0) -- (6.5,1.4);
\draw(7.5,0) -- (6.5,2.1);
\draw(6.5,1.4) -- (6.5,2.1);
\draw(6.5,2.1) -- (7.5,2.5);
\draw(7.5,2.5) -- (8.5,1.75);
\draw(8.5,1.75) -- (7.5,1);
\draw(6.5,1.4) -- (7.5,1);

\node at (7.5,-0.5) {$1$};

\end{tikzpicture}
\caption{A tree in $\icg(2)$, a one-loop graph in $\icg(3)$ and a two-loop graph in $\icg(1)$.}\label{figure:icgs}
\end{figure}
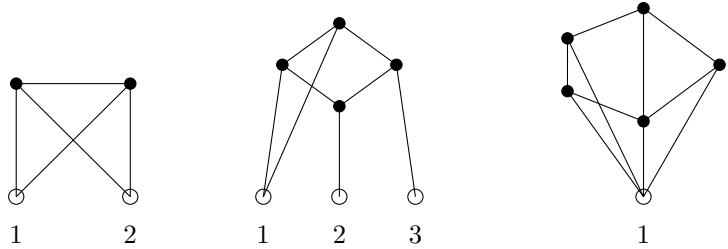

\subsection{Cohomology of the two-loop part}

In \cite{AroneTurchin2015}, G. Arone and V. Turchin defined a graph complex $(\mathcal{E}_{\pi}^{m,N},\partial)$ ($m,N\in \mathbb{N}$) which provides some insight on the rational homotopy type of the spaces of long embeddings, i.e. a certain kind of embeddings $\mathbb{R}^m\hookrightarrow \mathbb{R}^N$. While the cohomology of the tree- and one-loop part of this complex was computed in \cite{AroneTurchin2015}, the cohomology of the two-loop part was established in \cite{Turchin2014}. In this section, we recall the results on the two-loop part from which we will later deduce our main theorem.

\begin{remark}
It is clear from the definition given in \cite{Turchin2014} that, for $m=0$, $N=2$, the graph complex $(\mathcal{E}_\pi^{0,2},\partial)$ corresponds (up to some degree issues) to the complex $(\icg(1),d_0)$. The only difference between the two definitions is that while in $\mathcal{E}_\pi^{0,2}$ the degree of a graph is given by $\# \text{edges}-2\# \text{internal vertices}$, in $\icg(1)$, it will be $1-\# \text{edges}+2\# \text{internal vertices}$. Note however, that the way one draws graphs in the respective complexes is different, but the correspondence is obvious (see Figure \ref{figure:comparison}).

\begin{figure}[ht]
\centering
\begin{tikzpicture}

\node at (-3.5,0) {$1$};

\draw (0,0) circle (1.5cm);
\draw(-1.5,0) -- (1.5,0);

\draw[black,fill=black](0.79,1.28) circle (0.075cm);
\draw[black,fill=black](-0.79,1.28) circle (0.075cm);

\draw[black,fill=black](1.21,-0.88) circle (0.075cm);
\draw[black,fill=black](-1.21,-0.88) circle (0.075cm);
\draw[black,fill=black](0.46,-1.43) circle (0.075cm);
\draw[black,fill=black](-0.46,-1.43) circle (0.075cm);

\draw[black,fill=black](-1.5,0) circle (0.075cm);
\draw[black,fill=black](1.5,0) circle (0.075cm);

\draw(-3,0) circle (0.1cm);

\draw(-3,0) -- (-1.5,0);

\draw(0.79,1.28) -- (-3,0);
\draw(-0.79,1.28) -- (-3,0);

\draw(1.21,-0.88) -- (-3,0);
\draw(-1.21,-0.88) -- (-3,0);
\draw(0.46,-1.43) -- (-3,0);
\draw(-3,0) -- (-0.46,-1.43);

\node at (3,0) {$\longleftrightarrow$};

\draw (6,0) circle (1.5cm);
\draw(4.5,0) -- (7.5,0);

\draw[black,fill=black](6.79,1.28) circle (0.075cm);
\draw[black,fill=black](5.21,1.28) circle (0.075cm);

\draw[black,fill=black](7.21,-0.88) circle (0.075cm);
\draw[black,fill=black](4.79,-0.88) circle (0.075cm);
\draw[black,fill=black](6.46,-1.43) circle (0.075cm);
\draw[black,fill=black](5.54,-1.43) circle (0.075cm);

\draw[black,fill=black](4.5,0) circle (0.075cm);
\draw[black,fill=black](7.5,0) circle (0.075cm);

\draw(4.5,0) -- (4,0);

\draw(6.79,1.28) -- (7,1.73);
\draw(5.21,1.28) -- (5,1.73);

\draw(7.21,-0.88) -- (7.62,-1.18);
\draw(4.79,-0.88) -- (4.38,-1.18);
\draw(6.46,-1.43) -- (6.62,-1.9);
\draw(5.38,-1.9) -- (5.54,-1.43);

\end{tikzpicture}
\caption{The correspondence between graphs in $\icg(1)$ and $\mathcal{E}_\pi^{0,2}$.}\label{figure:comparison}
\end{figure}
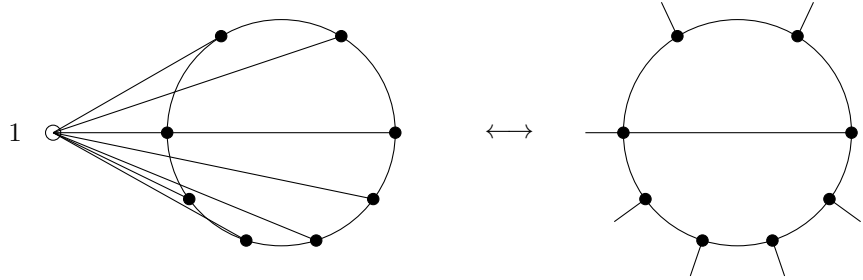\end{remark}

\begin{remark}
In \cite{Turchin2014} it is shown that the cohomology of the two-loop part of the complex $(\icg(1),d_0)\simeq(\mathcal{E}_\pi^{0,2},\partial)$ can be calculated by the quasi-isomorphic subcomplex $(C,d_0)$ given by
\begin{equation*}
0\rightarrow C_0 \rightarrow C_1\rightarrow C_2\rightarrow 0
\end{equation*}
where the graphs in $C_i$ are as depicted in Figure \ref{figure:Ccomplex}. In our notation, the cohomology of this complex is $E_1^{2,\bullet}(1)=H^{2+\bullet}(\mf^2\icg(1)/\mf^3\icg(1),d_0)$.
\end{remark}

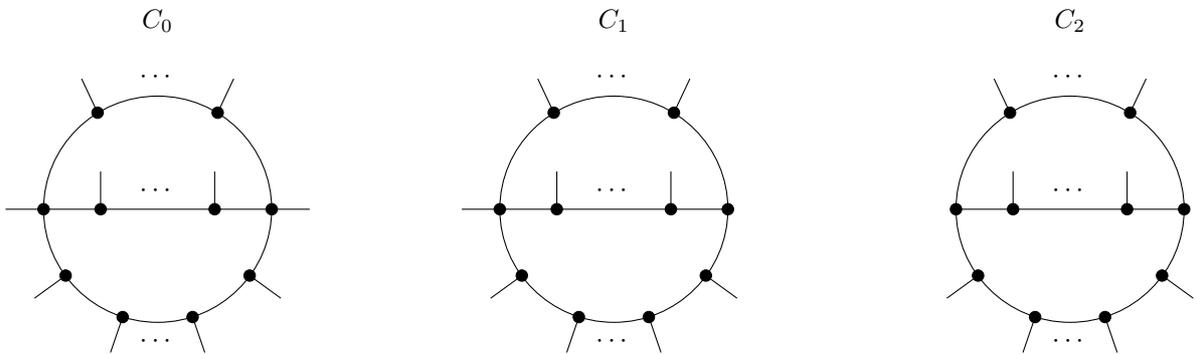
\begin{figure}[ht]
\centering
\begin{tikzpicture}

\node at (0,2.5) {$C_0$};
\node at (6,2.5) {$C_1$};
\node at (12,2.5) {$C_2$};

\node at (0,1.75) {$\cdots$};
\node at (0,0.25) {$\cdots$};
\node at (0,-1.75) {$\cdots$};

\draw (0,0) circle (1.5cm);
\draw(-1.5,0) -- (1.5,0);

\draw[black,fill=black](0.79,1.28) circle (0.075cm);
\draw[black,fill=black](-0.79,1.28) circle (0.075cm);

\draw[black,fill=black](0.75,0) circle (0.075cm);
\draw[black,fill=black](-0.75,0) circle (0.075cm);

\draw[black,fill=black](1.21,-0.88) circle (0.075cm);
\draw[black,fill=black](-1.21,-0.88) circle (0.075cm);
\draw[black,fill=black](0.46,-1.43) circle (0.075cm);
\draw[black,fill=black](-0.46,-1.43) circle (0.075cm);

\draw[black,fill=black](-1.5,0) circle (0.075cm);
\draw[black,fill=black](1.5,0) circle (0.075cm);

\draw(1.5,0) -- (2,0);
\draw(-2,0) -- (-1.5,0);

\draw(0.75,0) -- (0.75,0.5);
\draw(-0.75,0) -- (-0.75,0.5);

\draw(0.79,1.28) -- (1,1.73);
\draw(-0.79,1.28) -- (-1,1.73);

\draw(1.21,-0.88) -- (1.62,-1.18);
\draw(-1.21,-0.88) -- (-1.62,-1.18);
\draw(0.46,-1.43) -- (0.62,-1.9);
\draw(-0.62,-1.9) -- (-0.46,-1.43);


\node at (0+6,1.75) {$\cdots$};
\node at (6,0.25) {$\cdots$};
\node at (6,-1.75) {$\cdots$};

\draw (6,0) circle (1.5cm);
\draw(4.5,0) -- (7.5,0);

\draw[black,fill=black](6.79,1.28) circle (0.075cm);
\draw[black,fill=black](5.21,1.28) circle (0.075cm);

\draw[black,fill=black](6.75,0) circle (0.075cm);
\draw[black,fill=black](-0.75+6,0) circle (0.075cm);

\draw[black,fill=black](7.21,-0.88) circle (0.075cm);
\draw[black,fill=black](4.79,-0.88) circle (0.075cm);
\draw[black,fill=black](6.46,-1.43) circle (0.075cm);
\draw[black,fill=black](5.54,-1.43) circle (0.075cm);

\draw[black,fill=black](4.5,0) circle (0.075cm);
\draw[black,fill=black](7.5,0) circle (0.075cm);

\draw(4.5,0) -- (4,0);

\draw(6.79,1.28) -- (7,1.73);
\draw(5.21,1.28) -- (5,1.73);

\draw(0.75+6,0) -- (0.75+6,0.5);
\draw(-0.75+6,0) -- (-0.75+6,0.5);

\draw(7.21,-0.88) -- (7.62,-1.18);
\draw(4.79,-0.88) -- (4.38,-1.18);
\draw(6.46,-1.43) -- (6.62,-1.9);
\draw(5.38,-1.9) -- (5.54,-1.43);


\node at (0+12,1.75) {$\cdots$};
\node at (12,0.25) {$\cdots$};
\node at (12,-1.75) {$\cdots$};

\draw (12,0) circle (1.5cm);
\draw(4.5+6,0) -- (7.5+6,0);

\draw[black,fill=black](6.79+6,1.28) circle (0.075cm);
\draw[black,fill=black](5.21+6,1.28) circle (0.075cm);

\draw[black,fill=black](0.75+12,0) circle (0.075cm);
\draw[black,fill=black](-0.75+12,0) circle (0.075cm);

\draw[black,fill=black](7.21+6,-0.88) circle (0.075cm);
\draw[black,fill=black](4.79+6,-0.88) circle (0.075cm);
\draw[black,fill=black](6.46+6,-1.43) circle (0.075cm);
\draw[black,fill=black](5.54+6,-1.43) circle (0.075cm);

\draw[black,fill=black](4.5+6,0) circle (0.075cm);
\draw[black,fill=black](7.5+6,0) circle (0.075cm);

\draw(6.79+6,1.28) -- (7+6,1.73);
\draw(5.21+6,1.28) -- (11,1.73);

\draw(0.75+12,0) -- (0.75+12,0.5);
\draw(-0.75+12,0) -- (-0.75+12,0.5);

\draw(7.21+6,-0.88) -- (7.62+6,-1.18);
\draw(4.79+6,-0.88) -- (4.38+6,-1.18);
\draw(6.46+6,-1.43) -- (6.62+6,-1.9);
\draw(5.38+6,-1.9) -- (5.54+6,-1.43);

\end{tikzpicture}
\caption{Graphs in $C_0$, $C_1$, $C_2$.}\label{figure:Ccomplex}
\end{figure}

\begin{definition}
On a two-loop graph as in Figure \ref{figure:comparison} on the right, we will refer to the edges making up the $\theta$-shape as the three \emph{main edges}, and the short edges attached to the main edges (as well as the one attached to the vertices where the three main edges come together) as \emph{hair}. These correspond to the edges adjacent to the unique external vertex in $\icg(1)$ under the correspondence of $\mathcal{E}_\pi^{0,2}$ and $\icg(1)$.
\end{definition}

\begin{remark}
Following \cite{Turchin2014}, we encode such $\theta$-graphs by certain polynomials. Let the graph in $C_0$ with $k_1$ hair on the upper strand, $k_2$ on the middle one and $k_3$ on the lower one, and edges ordered as in Figure \ref{figure:ordering} correspond to $x^{k_1}y^{k_2}z^{k_3}$.

\begin{figure}[ht]
\centering
\begin{tikzpicture}


\node at (-1.75,0.25) {$0$};

\node at (-1.51,0.875) {$1$};
\node at (-1.125,1.95) {$2$};
\node at (0,1.75) {$3$};
\node at (1.125,1.95) {$4$};
\node at (1.51,0.875) {$5$};

\node at (-1.125,0.25) {$6$};
\node at (-0.75,0.75) {$7$};
\node at (-0.375,0.25) {$8$};
\node at (0,0.75) {$9$};
\node at (0.375,0.25) {$10$};
\node at (0.75,0.75) {$11$};
\node at (1.125,0.25) {$12$};

\node at (-1.66,-0.54) {$13$};
\node at (-1.82,-1.32) {$14$};
\node at (-1.03,-1.42) {$15$};
\node at (-0.7,-2.14) {$16$};
\node at (0,-1.75) {$17$};
\node at (0.7,-2.14) {$18$};
\node at (1.03,-1.42) {$19$};
\node at (1.82,-1.32) {$20$};
\node at (1.66,-0.54) {$21$};

\node at (1.75,0.25) {$22$};

\draw (0,0) circle (1.5cm);
\draw(-1.5,0) -- (1.5,0);

\draw[black,fill=black](0.79,1.28) circle (0.075cm);
\draw[black,fill=black](-0.79,1.28) circle (0.075cm);

\draw[black,fill=black](0.75,0) circle (0.075cm);
\draw[black,fill=black](-0.75,0) circle (0.075cm);
\draw[black,fill=black](0,0) circle (0.075cm);

\draw[black,fill=black](1.21,-0.88) circle (0.075cm);
\draw[black,fill=black](-1.21,-0.88) circle (0.075cm);
\draw[black,fill=black](0.46,-1.43) circle (0.075cm);
\draw[black,fill=black](-0.46,-1.43) circle (0.075cm);

\draw[black,fill=black](-1.5,0) circle (0.075cm);
\draw[black,fill=black](1.5,0) circle (0.075cm);

\draw(1.5,0) -- (2,0);
\draw(-2,0) -- (-1.5,0);

\draw(0.75,0) -- (0.75,0.5);
\draw(-0.75,0) -- (-0.75,0.5);
\draw(0,0) -- (0,0.5);

\draw(0.79,1.28) -- (1,1.73);
\draw(-0.79,1.28) -- (-1,1.73);

\draw(1.21,-0.88) -- (1.62,-1.18);
\draw(-1.21,-0.88) -- (-1.62,-1.18);
\draw(0.46,-1.43) -- (0.62,-1.9);
\draw(-0.62,-1.9) -- (-0.46,-1.43);
\end{tikzpicture}
\caption{A $\theta$-graph with this ordering on the set of edges corresponds to the polynomial $x^2y^3z^4$.}\label{figure:ordering}
\end{figure}
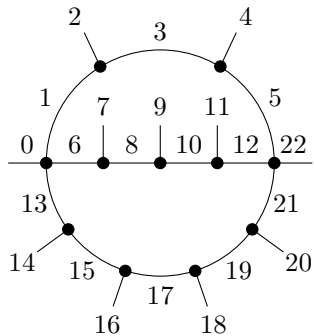

By abuse of notation, we identify $C_1$ and $C_2$ with the same type of polynomials. The different symmetry relations on graphs imply the following sign relations on our spaces of polynomials.
\end{remark}

\begin{lemma}[\cite{Turchin2014}, Lemma 4.2.]\label{Lemma:signs}
\begin{enumerate}
\item{Let $x^{k_1}y^{k_2}z^{k_3}\in C_0$. The symmetry with respect to the vertical line produces the sign
\begin{equation*}
(-1)^{k_1+k_2+k_3+1}.
\end{equation*}
Therefore, graphs in $C_0$ with an even number of hair on the main edges are zero.
}
\item{Let $x^{k_1}y^{k_2}z^{k_3}\in C_2$. The symmetry with respect to the vertical line produces the sign
\begin{equation*}
(-1)^{k_1+k_2+k_3}.
\end{equation*}
Therefore, graphs in $C_2$ with an odd number of hair are zero.
}
\item{Let $x^{k_1}y^{k_2}z^{k_3}\in C_i$, $i=0,1,2$. The $S_3$-action which permutes the three main edges acts by sign, i.e. for $\sigma\in S_3$,
\begin{equation*}
x^{k_1}y^{k_2}z^{k_3} \text{ is identified with } (-1)^{|\sigma|} x^{k_{\sigma(1)}}y^{k_{\sigma(2)}}z^{k_{\sigma(3)}}
\end{equation*}
where $|\sigma|$ denotes the parity of the permutation $\sigma$.
}
\end{enumerate}
\end{lemma}
\begin{proof}
The proof is an exercise in graphical calculus. It can be found in \cite{Turchin2014}.
\end{proof}
\begin{definition}
Let $\mathbb{K}[x,y,z]$ be the algebra of polynomials in the variables $x,y,z$. We say that a polynomial $p\in \mathbb{K}[x,y,z]$ is even (odd) if all of its monomials are
of even (odd) degree. We denote by  $\mathbb{K}[x,y,z]^{\text{even}}$ ($\mathbb{K}[x,y,z]^{\text{odd}}$) the space of even (odd) polynomials. Given a polynomial $p\in \poly$, we denote by $[p]_\text{even}$ ($[p]_\text{odd}$) its even (odd) part.
\end{definition}

\begin{remark}
Let the action of $S_3$ on $\poly$ be given on monomials by 
\begin{equation*}
\sigma . (x^{k_1}y^{k_2}z^{k_3}):=(-1)^{|\sigma|} x^{k_{\sigma(1)}}y^{k_{\sigma(2)}}z^{k_{\sigma(3)}}.
\end{equation*}
Using Lemma \ref{Lemma:signs} we can identify the components of the graph complex $(C,d_0)$ with the following spaces of coinvariants with respect to this action of the symmetric group.
\begin{align*}
C_0=&(\polyodd)_{S_3}\\
C_1=&\poly_{S_3}\\
C_2=&(\polyeven_{>0})_{S_3}
\end{align*}
where $\polyeven_{>0}$ denotes the space of even polynomials in which all monomials have strictly positive degree.
\end{remark}

\begin{lemma}[\cite{Turchin2014}, Lemma 5.6.]
The action of the differential $d_0$ translates under the above identifications to
\begin{align*}
d_0:C_0&\longrightarrow C_1\\
p(x,y,z)&\longmapsto 2(x+y+z)p(x,y,z)\\
d_0:C_1&\longrightarrow C_2\\
p(x,y,z)&\longmapsto \frac{1}{2}(x+y+z)(p(x,y,z)-p(-x,-y,-z))=[(x+y+z)p(x,y,z)]_{\text{even}}
\end{align*}
\end{lemma}

\begin{proof}
The differential $d_0:C_0\rightarrow C_1$ splits each of the two 4-valent vertices in three ways, and every time it adds an additional edge to one of the three main strands. In the space of polynomials, this corresponds to multiplying with $x+y+z$. The same operations is performed by $d_0:C_1\rightarrow C_2$ on the unique 4-valent vertex. In this case, by symmetry, all graphs with an odd number of hair will equal zero.
\end{proof}

\begin{lemma}
The cohomology of $(C,d_0)$ is
\begin{align*}
E_1^{2,-2}(1)=&H^0(C,d_0)=0\\
E_1^{2,-1}(1)=&H^1(C,d_0)=\left(\cfrac{\polyeven}{(x+y+z)\cap \polyeven}\right)_{S_3}\\
E_1^{2,0}(1)=&H^2(C,d_0)=\left(\cfrac{\polyeven_{>0}}{\im(d_0:C_1\rightarrow C_2)}\right)_{S_3}
\end{align*}
where $(x+y+z)$ denotes the ideal in $\mathbb{K}[x,y,z]$ generated by $x+y+z$.
\end{lemma}

\begin{proof}
The differential $d_0:C_0\rightarrow C_1$ is injective. Thus $H^0(C,d_0)=0$. Next, let $p\in C_1$ be such that $d_0p=\frac{1}{2}[(x+y+z)p]_{\text{even}}=0$. This is equivalent to $[p]_{\text{odd}}=0$, and therefore $p\in \polyeven$. The image of $d_0:C_0\rightarrow C_1$ is the ideal generated by $(x+y+z)$ which by symmetry we need to intersect with the algebra of even polynomials, and we obtain the result for $H^1(C,d_0)$. Since $d_0:C_2\rightarrow 0$, the formula for $H^2(C,d_0)$ follows by definition.
\end{proof}

\begin{definition}
Let $\gamma\in \icg(1)$. The number of edges adjacent to the unique external vertex in $\icg(1)$ is called the \emph{weight} of $\gamma$. For $i=1,2$, denote by $[H^i(C,d_0)]_k=[E_1^{2,1-i}(1)]_k$ the space of graphs of weight $k$ in $H^i(C,d_0)=E_1^{2,1-i}(1)$.
\end{definition}

\begin{remark}
In \cite{Turchin2014}, the authors work with invariants rather than with co-invariants to describe the cohomology. In particular, they give a generating set for $H^1(C,d_0)$ and $H^2(C,d_0)$ in terms of anti-symmetric polynomials. This enables them to compute the dimensions of the homogeneous weight components of these spaces (see Theorem \ref{Thm:dimensions} below). For our purposes, working with co-invariants seems to be more suitable. In fact, the generating set for $H^1(C,d_0)$ that we determine in Section \ref{section:basis} contains fewer elements than the one given in \cite{Turchin2014}.
\end{remark}

\begin{theorem}[\cite{Turchin2014}, Theorem 6.2.]\label{Thm:dimensions}
The dimensions of the aforementioned spaces are given by
\begin{align*}
\dim([H^1(C,d_0)]_k)=\dim([E_1^{2,-1}(1)]_k)&=  \begin{cases}
    0, & \text{for } k \text{ even} \\
   \lfloor\frac{k}{6}\rfloor, & \text{for } k \text{ odd}
  \end{cases}\\
\dim([H^2(C,d_0)]_k)=\dim([E_1^{2,0}(1)]_k)&=  \begin{cases}
   \lfloor\frac{k}{6}\rfloor, & \text{for } k \text{ even} \\
    0, & \text{for } k \text{ odd.}
  \end{cases}
\end{align*}
\end{theorem}

\begin{remark}
As vector spaces $E_1^{2,-1}(1)$ and $\mf^2\grt_1/\mf^3\grt_1$ decompose into their respective weight components. That is,
\begin{align*}
E_1^{2,-1}(1)=&\bigoplus\limits_{k\geq 6}{[E_1^{2,-1}(1)]_{k+1}}\\
\mf^2\grt_1/\mf^3\grt_1=&\bigoplus\limits_{k\geq 8}{[\mf^2\grt_1/\mf^3\grt_1]_k}.
\end{align*}
\end{remark}

\begin{corollary}
There is an isomorphism of vector spaces,
\begin{equation*}
\mf^2\grt_1/\mf^3\grt_1\cong E_1^{2,-1}(1).
\end{equation*}
\end{corollary}

\begin{proof}
From Remark \ref{rmk:relations} and Theorem \ref{Thm:dimensions}, we obtain for all even $k\geq 8$,
\begin{equation*}
\dim([E_1^{2,-1}(1)]_{k-1})=\lfloor \frac{k-1}{6}\rfloor=\lfloor \frac{k-2}{6}\rfloor=\dim([\mf^2\grt_1/\mf^3\grt_1]_k)
\end{equation*}
which implies $[\mf^2\grt_1/\mf^3\grt_1]_k\cong[E_1^{2,-1}(1)]_{k-1}$, and thus $\mf^2\grt_1/\mf^3\grt_1\cong E_1^{2,-1}(1)$.
\end{proof}

Still, there remains a choice on the isomorphism. Next, we determine a generating set for $E_1^{2,-1}(1)$ which will enable us to describe a particularly easy isomorphism.

\subsection{A generating family of $\theta$-graphs}\label{section:basis}

\begin{remark}\label{rmk:action}
The algebra homomorphism $\phi:\mathbb{K}[x,y,z] \rightarrow \mathbb{K}[x,y]$ which on generators is defined via $x\mapsto x$, $y\mapsto y$, $z\mapsto -x-y$, induces an isomorphism,
\begin{equation*}
\phi:\cfrac{\mathbb{K}[x,y,z]}{(x+y+z)}\rightarrow \mathbb{K}[x,y].
\end{equation*}
If we restrict to even polynomials, we still get an isomorphism,
\begin{equation*}
\phi:\cfrac{\mathbb{K}[x,y,z]^{\text{even}}}{(x+y+z)\cap \polyeven}\rightarrow \mathbb{K}[x,y]^{\text{even}}.
\end{equation*}
This enables us to define an $S_3$-action on $\mathbb{K}[x,y]$ via the formula,
\begin{equation*}
\sigma_*(x^{k_1} y^{k_2}):=\phi(\sigma.(x^{k_1}y^{k_2}))
\end{equation*}
for $\sigma \in S_3$. It is indeed a group action, since for $\sigma, \tau\in S_3$, we have
\begin{equation*}
(\sigma\tau)_*(x^{k_1} y^{k_2})=\phi((\sigma\tau).(x^{k_1} y^{k_2}))=\phi(\sigma.(\tau.(x^{k_1} y^{k_2})))=\sigma_*(\phi(\tau.(x^{k_1} y^{k_2})))=\sigma_*(\tau_*(x^{k_1} y^{k_2})),
\end{equation*}
and also 
\begin{equation*}
\id_*(x^{k_1} y^{k_2})=\phi(\id.(x^{k_1} y^{k_2}))=\phi(x^{k_1} y^{k_2})=x^{k_1} y^{k_2}.
\end{equation*}
\end{remark}

\begin{lemma}\label{lemma:equivariance}
The algebra homomorphism $\phi$ is $S_3$-equivariant, that is, for all $\sigma\in S_3$ and $p\in \mathbb{K}[x,y,z]^{\text{even}}$, we have 
\begin{equation*}
\sigma_* \phi(p)=\phi (\sigma.p).
\end{equation*}
\end{lemma}

\begin{proof}
First notice that for $f,g\in p[x,y,z]^{\text{even}}$ and $\sigma\in S_3$ we have, $\sigma.(fg)=(-1)^{|\sigma|}(\sigma.f)\cdot (\sigma.g).$ From this, we deduce $\sigma.(x^{k_1})=(-1)^{|\sigma|(k_1-1)} (\sigma.x)^{k_1}$ and analogously for $y$ and $z$. Moreover, it is easy to check by direct computation that $\phi(\sigma.(-x-y))=\phi(\sigma.z)$ for all $\sigma \in S_3$. Finally, equivariance follows from,
\begin{align*}
&\sigma_*\phi(x^{k_1}y^{k_2}z^{k_3})=\sigma_*(x^{k_1}y^{k_2}(-x-y)^{k_3})=\sum\limits_{j=0}^{k_3}{(-1)^{k_3}{k_3\choose j}\sigma_*(x^{k_1+j}y^{k_2+k_3-j})}\\
=&\sum\limits_{j=0}^{k_3}{(-1)^{k_3}{k_3\choose j}\phi(\sigma.(x^{k_1+j}y^{k_2+k_3-j}))}=\phi(\sigma.\sum\limits_{j=0}^{k_3}{(-1)^{k_3}{k_3\choose j}x^{k_1+j}y^{k_2+k_3-j}})\\
=&\phi(\sigma.(x^{k_1} y^{k_2} (-x-y)^{k_3}))=\phi(\sigma.(x^{k_1})\sigma.(y^{k_2})\sigma.((-x-y)^{k_3}))=\phi(\sigma.(x^{k_1})\sigma.(y^{k_2}))\phi(\sigma.((-x-y)^{k_3}))\\
=&(-1)^{|\sigma|(k_3-1)}\phi(\sigma.(x^{k_1})\sigma.(y^{k_2}))\phi(\sigma.(-x-y))^{k_3}
=(-1)^{|\sigma|(k_3-1)}\phi(\sigma.(x^{k_1})\sigma.(y^{k_2}))\phi(\sigma.z)^{k_3}\\
=&\phi(\sigma.(x^{k_1})\sigma.(y^{k_2}))\phi(\sigma.(z^{k_3}))=\phi(\sigma.(x^{k_1})\sigma.(y^{k_2})\sigma.(z^{k_3}))=\phi(\sigma.(x^{k_1}y^{k_2}z^{k_3})).
\end{align*}
Thus, by linearity, $\sigma_*\phi(p)=\phi(\sigma.p)$ holds for all $p\in \mathbb{K}[x,y,z]^{even}$ and all $\sigma \in S_3$.
\end{proof}

Since $\phi$ preserves the $S_3$-action, we may take coinvariants on both sides to obtain an isomorphism
\begin{equation*}
\phi:\left(\cfrac{\mathbb{K}[x,y,z]^{\text{even}}}{(x+y+z)\cap \polyeven}\right)_{S_3}\longrightarrow \left(\mathbb{K}[x,y]^{\text{even}}\right)_{S_3}.
\end{equation*}
Next, consider the subalgebra $A$ of $\mathbb{K}[x,y]^{\text{even}}$ generated by monomials $x^ay^b$ with $a\leq b$ even, i.e. 
\begin{equation*}
A:=\text{span}( x^a y^b | 0\leq a\leq b \text{ even})
\end{equation*}
We define recursively a linear map,
\begin{align*}
\psi:\mathbb{K}[x,y]^{\text{even}}&\longrightarrow A\\
x^ay^b&\mapsto 
\begin{cases}
0, & \text{if } a=0 \text{ or } b=0, \\
x^a y^b, & \text{if } a\leq b \text{ both even},\\
\frac{-1}{a+1}\psi\left(x^{a+1} y^{b-1}+\sum\limits_{\substack{j=1\\ j\neq a}}^{a+1}{a+1 \choose j} x^j y^{a+b-j}\right), & \text{if } a\leq b \text{ both odd},\\
-\psi(x^b y^a) & \text{if } a\geq b. 
\end{cases}
\end{align*}
Note that for $a=1$, $b\geq 3$ odd, we have $\psi(xy^b)=-\psi(x^2y^{b-1})=-x^2y^{b-1}$. Moreover, for $a\leq b$ both odd, $\psi(x^ay^b)$ may be reformulated as
\begin{equation}\label{eq:psi}
\psi(x^ay^b)=\frac{-1}{a+1}\left(\psi(x^{a+1} y^{b-1})+\sum\limits_{\substack{j=2 \\ j\text{ even}}}^{a+1}{a+1 \choose j} \psi(x^j y^{a+b-j})+\sum\limits_{\substack{j=1 \\ j\text{ odd}}}^{a-2}{a+1 \choose j} \psi(x^j y^{a+b-j})\right).
\end{equation}
The first term and the first sum are obviously well-defined as $\psi$ acts (up to sign) as the identity on monomials with even exponents while the second sum is well-defined by induction on $a$ (the exponent of $x$ is always strictly smaller than $a$). Let us denote by $B$ the subspace of $A$ generated by the image of the relations on $\mathbb{K}[x,y]^{\text{even}}_{S_3}$ under $\psi$, that is,
\begin{equation*}
B:=\text{span}_\mathbb{K}(\psi(\sigma_*v)-\psi(v)|\sigma\in S_3 \text{, } v\in\mathbb{K}[x,y]^{\text{even}} ).
\end{equation*}
In this way, the map $\psi$ induces a well-defined surjection,
\begin{equation*}
\psi:\mathbb{K}[x,y]^{\text{even}}_{S_3}\longrightarrow A/B.
\end{equation*}

\begin{lemma}
Let $i:A/B\rightarrow\mathbb{K}[x,y]^{\text{even}}_{S_3}$, $x^ay^b\mapsto x^ay^b$, $0\leq a\leq b$ even. We have $\psi \circ i=\id$ and $i\circ \psi=\id$. Thus, as vector spaces,
\begin{equation*}
A/B\cong \mathbb{K}[x,y]^{\text{even}}.
\end{equation*}
\end{lemma}

\begin{proof}
Since $\psi$ is the identity on $x^ay^b$, $0\leq a\leq b$ even, clearly $\psi\circ i=\id$. For the other composition, we proceed by induction. For $a=1$, $b\geq 3$ odd, we have $i\psi(xy^b)=-i(x^2y^{b-1})=-x^2y^{b-1}$ which can easily be checked to equal $xy^{b}$ in $\mathbb{K}[x,y]^{\text{even}}_{S_3}$. Moreover, whenever $a$ and $b$ are even, $i\psi(x^ay^b)=x^ay^b\in \mathbb{K}[x,y]^{\text{even}}_{S_3}$. Therefore, for $a\leq b$ odd,
\begin{align*}
i\psi(x^ay^b)=&\frac{-1}{a+1}\left(i\psi(x^{a+1} y^{b-1})+\sum\limits_{\substack{j=2 \\ j\text{ even}}}^{a+1}{a+1 \choose j} i\psi(x^j y^{a+b-j})+\sum\limits_{\substack{j=1 \\ j\text{ odd}}}^{a-2}{a+1 \choose j} i\psi(x^j y^{a+b-j})\right)\\
=&\frac{-1}{a+1}\left(x^{a+1} y^{b-1}+\sum\limits_{\substack{j=2 \\ j\text{ even}}}^{a+1}{a+1 \choose j} x^j y^{a+b-j}+\sum\limits_{\substack{j=1 \\ j\text{ odd}}}^{a-2}{a+1 \choose j} x^j y^{a+b-j}\right)\\
=&\frac{-1}{a+1}\left(x^{a+1}y^{b-1}+(-x-y)^{a+1}y^{b-1}-(a+1)x^ay^b\right)\\
=&\frac{-1}{a+1}\left(x^{a+1}y^{b-1}-(13)_*(x^{a+1}y^{b-1})-(a+1)x^ay^b\right)=x^ay^b\in \mathbb{K}[x,y]^{\text{even}}_{S_3},
\end{align*}
where in the first line we used that by induction all compositions $i\circ\psi$ appearing on the right hand side equal $\id$.
\end{proof}

\begin{remark}\label{rmk:generate}
Under the chain of identifications,
\begin{equation*}
A/B\cong \mathbb{K}[x,y]^{\text{even}}_{S_3}\cong \left(\cfrac{\mathbb{K}[x,y,z]^{\text{even}}}{(x+y+z)\cap \polyeven}\right)_{S_3}\cong E_1^{2,-1}(1)
\end{equation*}
a basic element $x^{2i}y^{2j}\in A/B$ corresponds to the theta graph $\theta_{2i,2j}\in E_1^{2,-1}(1)$ that has only two of its main strands decorated by $2i$ and $2j$ hair, respectively (see Figure \ref{figure:thetaicg}). In particular, we find that these graphs generate $E_1^{2,-1}(1)$.

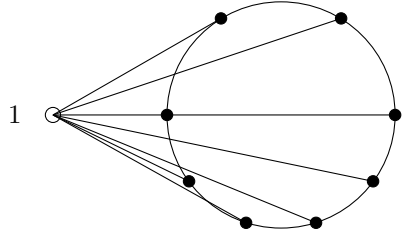
\begin{figure}[ht]
\centering
\begin{tikzpicture}

\draw (0,0) circle (1.5cm);
\draw(-1.5,0) -- (1.5,0);

\draw[black,fill=black](0.79,1.28) circle (0.075cm);
\draw[black,fill=black](-0.79,1.28) circle (0.075cm);

\draw[black,fill=black](1.21,-0.88) circle (0.075cm);
\draw[black,fill=black](-1.21,-0.88) circle (0.075cm);
\draw[black,fill=black](0.46,-1.43) circle (0.075cm);
\draw[black,fill=black](-0.46,-1.43) circle (0.075cm);

\draw[black,fill=black](-1.5,0) circle (0.075cm);
\draw[black,fill=black](1.5,0) circle (0.075cm);

\draw(-3,0) circle (0.1cm);

\draw(-3,0) -- (-1.5,0);

\draw(0.79,1.28) -- (-3,0);
\draw(-0.79,1.28) -- (-3,0);

\draw(1.21,-0.88) -- (-3,0);
\draw(-1.21,-0.88) -- (-3,0);
\draw(0.46,-1.43) -- (-3,0);
\draw(-3,0) -- (-0.46,-1.43);

\node at (-3.5,0) {$1$};

\end{tikzpicture}
\caption{The graph $\theta_{2,4}$.}\label{figure:thetaicg}
\end{figure}

\end{remark}

\subsection{Main theorem and relations among $\theta$-graphs}
Our main result states the following.

\begin{theorem}\label{Theorem:Main}
Let $\theta_{2i,2j}\in E_1^{2,-1}(1)$ denote the theta graph that has only two of its main strands decorated by $2i$ and $2j$ hair, respectively. The map
\begin{align*}
\Phi:\mf^2\grt_1/\mf^3\grt_1&\rightarrow E_1^{2,-1}(1)\\
\{\sigma_{2i+1},\sigma_{2j+1}\} \mod \mf^3\grt_1&\mapsto \theta_{2i,2j}.
\end{align*}
is an isomorphism of vector spaces.
\end{theorem}

\begin{proof}
We show that $\Phi$ is well-defined, injective and surjective. It is a map degree $-1$ with respect to the weight grading. It will be enough to check its properties on elements of homogeneous degree. Let $k\geq 8$ and
\begin{equation*}
L:=\sum\limits_{i=1}^{\frac{k-4}{2}}a_i \{\sigma_{2i+1},\sigma_{k-1-2i}\}=0\mod \mf^3\grt_1.
\end{equation*}
By Remark \ref{rmk:identification}, this is equivalent to $G=\sum\limits_{i=1}^{\frac{k-4}{2}} a_i (\alpha-\beta)^{2i}(\beta-\gamma)^{k-2-2i}$ satisfying $G+(13).G=0$ and $G+(123).G+(132).G=0$. Note that the first equation implies that $a_i=-a_{k/2-1-i}$, and thus we may write $G$ as,
\begin{equation*}
G=\sum\limits_{i=1}^{\lfloor\frac{k-4}{2}\rfloor} a_i \left((\alpha-\beta)^{2i}(\beta-\gamma)^{k-2-2i}-(\alpha-\beta)^{k-2-2i}(\beta-\gamma)^{2i}\right).
\end{equation*}
Also, remark that $G\in \mathbb{K}[\alpha-\beta,\beta-\gamma]^{\text{even}}$ and that there is an isomorphism of algebras,
\begin{equation*}
\cfrac{\mathbb{K}[x,y,z]^{\text{even}}}{(x+y+z)\cap \polyeven}\longrightarrow \mathbb{K}[\alpha-\beta,\beta-\gamma]^{\text{even}},
\end{equation*}
induced by the surjective map $\mathbb{K}[x,y,z]^{\text{even}}\rightarrow \mathbb{K}[\alpha-\beta,\beta-\gamma]^{\text{even}}$, $x\mapsto \alpha-\beta$, $y\mapsto \beta-\gamma$, $z\mapsto \gamma-\alpha$. Note that the isomorphism is not $S_3$-equivariant. Consider the preimage of the second equation for $G$ under this isomorphism. It is represented by,
\begin{equation*}
\sum\limits_{i=1}^{\lfloor\frac{k-4}{2}\rfloor} a_i \sum\limits_{\sigma\in S_3} \sigma. x^{2i}y^{k-2-2i}=0.
\end{equation*}
This is an equation of $S_3$-invariant elements, and thus the linear combination on the left is zero also in the space of invariants,
\begin{equation*}
\left(\frac{\mathbb{K}[x,y,z]^{\text{even}}}{(x+y+z)\cap \polyeven}\right)^{S_3},
\end{equation*}
which in turn is isomorphic to the space of coinvariants. Moreover, under the corresponding isomorphism, the above equation reads,
\begin{equation*}
3! \sum\limits_{i=1}^{\lfloor\frac{k-4}{2}\rfloor} a_i x^{2i}y^{k-2-2i}=0\in \left(\frac{\mathbb{K}[x,y,z]^{\text{even}}}{(x+y+z)\cap \polyeven}\right)_{S_3}\cong E_1^{2,-1}(1).
\end{equation*}
Hence, since $\theta_{2i,k-2-2i}=-\theta_{k-2-2i,2i}$,
\begin{equation*}
\Phi(L)=\sum\limits_{i=1}^{\frac{k-4}{2}} a_i \theta_{2i,k-2-2i}=\sum\limits_{i=1}^{\frac{k-4}{2}}a_i (\theta_{2i,k-2-2i}-\theta_{k-2-2i,2i})=2 \cdot \sum\limits_{i=1}^{\lfloor\frac{k-4}{2}\rfloor} a_i \theta_{2i,k-2-2i}=0.
\end{equation*}
In fact, since we are only dealing with isomorphisms, the argument above may be traced backwards to imply the injectivity of $\Phi$. Since the graphs $\theta_{2i,2j}$ generate $E_1^{2,-1}(1)$, $\Phi$ is clearly surjective.
\end{proof}

Remark \ref{rmk:generate} and Theorem \ref{Theorem:Main} establish the following equivalence. For even $k\geq 8$,
\begin{equation*}
\sum\limits_{i=1}^{\lfloor\frac{k-4}{2}\rfloor} a_i x^{2i}y^{k-2-2i}=0\in A/B \Leftrightarrow \sum\limits_{i=1}^{\lfloor\frac{k-4}{2}\rfloor} a_i \{\sigma_{2i+1},\sigma_{k-1-2i}\}=0 \mod \mf^3\grt_1.
\end{equation*}
On the left hand side, finding non-trivial coefficients $a_i$ amounts to finding elements which are zero in $\mathbb{K}[x,y]^{\text{even}}_{S_3}$, but are sent to a non-trivial linear combination in $A$ by $\psi$. Natural candidates which might satisfy this condition are expressions of the form $\phi(x^a y^b z^a)=x^ay^b (-x-y)^a$ where $2a+b=k$. Their image $\psi(x^a y^b (-x-y)^a)$ might be equal to zero already in $A$. If not, however, this will produce a relation in $A/B$. Here is a list for the first few relations obtained in this way.
\begin{center}
\begin{tabular}{ c c }
&\\
  $k=10:$ & $0=\psi(x^3y^4(-x-y)^3)=\frac{1}{2} (-3 x^4 y^6 + x^2 y^8) $ \\
   &  \\
  $k=14:$ & $ 0=\psi(x^4y^6(-x-y)^4)=\frac{1}{3}(11 x^6 y^8 - 7 x^4 y^{10} + 2 x^2 y^{12})$ \\
  & \\
  $k=16:$ & $0=\psi(x^5y^6(-x-y)^5)=\frac{1}{12}(26 x^6 x^{10} -25 x^4 y^{12} +8 x^2 y^{14})$ \\
  &\\
  $k=18:$ & $0=\psi(x^5y^8(-x-y)^5)=\frac{1}{2}(-13 x^8 y^{10} + 14 x^6 y^{12} - 10 x^4 y^{14} + 3 x^2 y^{16})$\\
  &\\
  $k=20:$ & $0=\psi(x^6y^8(-x-y)^6)=\frac{1}{10}(-85 x^8 y^{12} + 136 x^6 y^{14} - 105 x^4 y^{16} + 32 x^2 y^{18})$\\
  &
\end{tabular}
\end{center}
These yield linear relations for the corresponding two-loops graphs $\theta_{2i,2j}$. Via the isomorphism $\Phi$, we recover the relations from Remark \ref{rmk:relations} for the brackets $\{\sigma_{2i+1},\sigma_{2j+1}\}$ in $\mf^2\grt_1/\mf^3\grt_1$. The coefficients in the list below agree with L. Schneps' calculations \cite{Schneps2006}.
\begin{center}
\begin{tabular}{ c c }
&\\
  $k=12:$ & $0=-3 \{\sigma_{5},\sigma_{7}\}+ \{\sigma_{3},\sigma_{9}\} \mod \mf^3\grt_1$ \\
   &  \\
  $k=16:$ & $ 0=11 \{\sigma_{7},\sigma_{9}\} - 7 \{\sigma_{5},\sigma_{11}\} + 2 \{\sigma_{3},\sigma_{13}\}\mod \mf^3\grt_1$ \\
  & \\
  $k=18:$ & $0=26 \{\sigma_{7},\sigma_{11}\} -25 \{\sigma_{5},\sigma_{13}\} +8\{\sigma_{3},\sigma_{15}\} \mod \mf^3\grt_1$ \\
  &\\
  $k=20:$ & $0=-13\{\sigma_{9},\sigma_{11}\} + 14 \{\sigma_{7},\sigma_{13}\}- 10\{\sigma_{5},\sigma_{15}\}  + 3\{\sigma_{3},\sigma_{17}\} \mod \mf^3\grt_1$\\
  &\\
  $k=22:$ & $0=-85 \{\sigma_{9},\sigma_{13}\}+ 136 \{\sigma_{7},\sigma_{15}\} - 105\{\sigma_{5},\sigma_{17}\}  + 32\{\sigma_{3},\sigma_{19}\} \mod \mf^3\grt_1$\\
\end{tabular}
\end{center}
In weight $k=24$, we expect, from Goncharov's \cite{Goncharov2001}, and Ihara and Takao's \cite{Ihara2002} formula, two independent linear relations for the brackets $\{\sigma_{2i+1},\sigma_{24-1-2i}\}$ modulo $\mf^3\grt_1$. These can be found by calculating for instance $\psi(x^6y^{10}(-x-y)^6)$ and $\psi(x^7y^8(-x-y)^7)$. The relations in $\mf^2\grt_1/\mf^3\grt_1$ obtained in this way are,
\begin{align*}
20\{\sigma_{11},\sigma_{13}\}-33\{\sigma_{9},\sigma_{15}\}+44\{\sigma_{7},\sigma_{17}\}-33\{\sigma_{5},\sigma_{19}\}+10\{\sigma_{3},\sigma_{21}\}=&0 \mod \mf^3\grt_1\\
-672\{\sigma_{11},\sigma_{13}\}+915\{\sigma_{9},\sigma_{15}\}-1106\{\sigma_{7},\sigma_{17}\}+805\{\sigma_{5},\sigma_{19}\}-242\{\sigma_{3},\sigma_{21}\}=&0\mod \mf^3\grt_1.
\end{align*}

\section{Relations to M. Kontsevich's graph complex}
We claim that the equivalence between $\grt_1$ in depth two modulo higher depth and the cohomology of the space of two-loop graphs is not a coincidence. For this, let us recall one further graph complex, $\GC$.

\subsection{Definitions}
The graph complex $\GC$ is a variant of M. Kontsevich's graph complex (\cite{Kontsevich1993},\cite{Kontsevich1994},\cite{Kontsevich1997}). We follow T. Willwacher's paper \cite{Willwacher2014}. 
\begin{definition}
Let  $\Gamma$ be an undirected graph with $N$ labeled vertices and $k$ edges satifying the following properties:
\begin{enumerate}
	\item{All vertices have valence at least three.}
	\item{There is a linear order on the set of edges.}
	\item{$\Gamma$ has no simple loops.}
\end{enumerate}
We denote by $\Gra(N,k)$ the graded vector space spanned by isomorphism classes of connected graphs satisfying the conditions above, modulo the relation $\Gamma\cong(-1)^{|\sigma|}\Gamma^{\sigma}$, where $\Gamma^\sigma$ differs from $\Gamma$ just by a permutation $\sigma\in S_k$ on the order of the edges. The degree of such a graph $\Gamma$ is given by
\begin{equation*}
\deg_{\Gra}\Gamma=-k.
\end{equation*}
\end{definition}
Set,
\begin{equation*}
\Gra(N):=\bigoplus\limits_{k\geq 0}\Gra(N,k).
\end{equation*}
The collection $\{\Gra(N)\}_{N\geq 1}$ naturally defines an operad $\Gra$ in the category of graded vector spaces. For $\Gamma\in \Gra(N)$, the $S_N$-action permutes the labels of the vertices. For $r,s\geq 1$, $\Gamma_1\in \Gra(r)$ and $\Gamma_2\in \Gra(s)$, the operadic composition $\Gamma_1\circ_j\Gamma_2\in \Gra(r+s-1)$ is given by inserting the graph $\Gamma_2$ at vertex $j$ of $\Gamma_1$ and summing over all ways of reconnecting the edges incident to vertex $j$ in $\Gamma_1$ to vertices of $\Gamma_2$. As in the case of $\icg$, we ask that the order on the set of edges of $\Gamma_1\circ \Gamma_2$ is such that all edges of $\Gamma_1$ come before those of $\Gamma_2$ while the respective orderings are left unaltered. Next, define,
\begin{equation*}
\GC:=\bigoplus\limits_{N\geq 1} \left( \Gra(N)[2-2N]\right)^{S_N}.
\end{equation*}
The space $\GC$ carries the structure of a differential graded Lie algebra. The degree of a graph $\Gamma\in \GC$ with $k$ edges and $N$ vertices is
\begin{equation*}
\deg\Gamma=-2-k+2N.
\end{equation*}
For the Lie bracket, consider the operadic pre-Lie product on $\Gra$,
\begin{equation*}
\Gamma_1\circ\Gamma_2=\sum\limits_{j=1}^{r} \Gamma_1\circ_j\Gamma_2.
\end{equation*}
Using this, the Lie bracket on $\GC$ is defined on homogeneous elements via,
\begin{equation*}
[\Gamma_1,\Gamma_2]:=\Gamma_1\circ\Gamma_2-(-1)^{\deg \Gamma_1\cdot \deg\Gamma_2}\Gamma_2\circ\Gamma_1.
\end{equation*}
The differential $d$ is given by vertex splitting, where again we ask that the newly created edge is placed last in the ordering of the edges. 

We define the following descending filtration on $\GC$.
\begin{equation*}
\mf^p\GC:=\{ \gamma \in \GC| \# \text{vertices} - \text{maximal valence} \geq p\}
\end{equation*}
where $p\geq 1$. Note that $\mf^1\GC=\GC$.

\begin{lemma}
The filtration is compatible with the differential graded Lie algebra structure on $\GC$, i.e. for all $p,q\geq 1$,
\begin{align*}
d(\mf^p\GC)\subset& \mf^p\GC\\
[\mf^p\GC,\mf^q\GC]\subset& \mf^{p+q}\GC
\end{align*}
Thus, $\mf^pH^0(\GC):=\{ \Gamma \in H^0(\GC)| \# \text{vertices} - \text{maximal valence} \geq p\}$ defines a descending filtration on $H^0(\GC)$.
\end{lemma}

\begin{proof}
Let $\Gamma\in \mf^{p}\GC$, $m_{\Gamma}:=\text{maximal valence of } \Gamma$ and $v_\Gamma:=\#\text{vertices of } \Gamma$. Note that the differential decreases the valence of a vertex at least by one while also creating one additional vertex (whose valence will be maximally equal to the valence of the splitted vertex minus one). Therefore, $m_{d\Gamma}\leq m_{\Gamma}$ and $v_{d\Gamma}=v_\Gamma+1$, and,
\begin{equation*}
v_{d\Gamma}-m_{d\Gamma}\geq v_\Gamma+1-m_\Gamma\geq p+1\geq p.
\end{equation*}
Let $\Gamma_1\in \mf^{p_1}\GC$ and $\Gamma_2\in \mf^{p_2}\GC$. For $i=1,2$, set $m_i:=\text{maximal valence of } \Gamma_i$ and $v_i:=\#\text{vertices of } \Gamma_i$. Denote by $m:=\#\text{maximal valence of } [\Gamma_1,\Gamma_2]$ and $v:=\#\text{vertices of } [\Gamma_1,\Gamma_2]$. Note that the insertion operation $\Gamma_1\circ \Gamma_2$ decreases the total number of vertices by 1. Hence, $v=v_1+v_2-1$. On the other hand, assuming the vertices of maximal valence are labeled by 1, the expression,
\begin{equation*}
\Gamma_1\circ_1\Gamma_2 -\Gamma_2\circ_1\Gamma_1,
\end{equation*}
produces two graphs with maximal valence $m_1+m_2$. One is obtained when the vertex of maximal valence in $\Gamma_1$ is replaced by $\Gamma_2$ and all ``loose" edges are reattached to the vertex of maximal valence in $\Gamma_2$. The other stems from an analogous construction, but with the roles of $\Gamma_1$ and $\Gamma_2$ exchanged. Luckily, due to symmetry, these two terms always cancel each other out. All other insertions yield $m\leq m_1+m_2-1$, and thus
\begin{equation*}
v-m\geq v_1+v_2-1-(m_1+m_2-1)\geq p_1+p_2,
\end{equation*}
from which the statement follows.
\end{proof}

\begin{remark}\label{remark:irreducible}
A graph in $\GC$ is called 1-vertex irreducible if it stays connected after deletion of any single vertex.
Denote by $\GC^{1-vi}$ the subcomplex of $(\GC,d)$ spanned by 1-vertex irreducible graphs. As shown in \cite{Conant2005}, the subcomplex $\GC^{1-vi}$ is quasi-isomorphic to $\GC$. We may thus assume that all cohomology classes we work with are represented by 1-vertex irreducible graphs.
\end{remark}

\subsection{Compatibility with T. Willwacher's isomorphism}
The graph complex $\GC$ is related to the Grothendieck-Teichm\"uller Lie algebra via its degree zero cohomology. More precisely, T. Willwacher's important result states the following.

\begin{theorem}[\cite{Willwacher2014}, Theorem 1.1.]
There is an isomorphism of Lie algebras,
\begin{equation*}
\omega: H^0(\GC)\overset{\cong}{\longrightarrow} \grt_1. 
\end{equation*}
\end{theorem}

The aim of this section is to prove the following statement.

\begin{theorem}\label{thm:compatible}
T. Willwacher's isomorphism $\omega$ induces an isomorphism,
\begin{equation*}
\mf^2H^0(\GC)/\mf^3H^0(\GC) \longrightarrow \mf^2\grt_1/\mf^3\grt_1.
\end{equation*}
\end{theorem}

As a consequence, we obtain identifications,
\begin{equation*}
\mf^2H^0(\GC)\mf^3H^0(\GC) \cong \mf^2\grt_1\mf^3\grt_1 \cong E_1^{2,-1}(1).
\end{equation*}
Conjecturally, $\omega$ should induce an isomorphism on the whole associated graded. Moreover, the hope is that the degree $p$ part of the associated graded may be identified with $E_1^{p,-p+1}(1)$, that is, the cohomology of $p$-loop graphs in $\icg(1)$ of degree $1$. This would yield the following isomorphisms,
\begin{equation*}
\mf^pH^0(\GC)/\mf^{p+1}H^0(\GC) \cong E_1^{p,-p+1}(1) \cong \mf^p\grt_1/\mf^{p+1}\grt_1,
\end{equation*}
which in turn would allow us to study the Grothendieck-Teichm\"uller Lie algebra's relations in higher depth by computing the cohomology of the space of higher loop order graphs in $\icg(1)$.

To start, recall from (\cite{Willwacher2014}, Proposition 9.1.) that all representatives of the cohomology classes $s_{2k+1}$ in $H^0(\GC)$ corresponding to the conjectural generators $\sigma_{2k+1}$ under the isomorphism $\omega$ have non-vanishing coefficient in front of the wheel graph $w_{2k+1}$ with $2k+1$ spokes (see Figure \ref{figure:wheel graphs}).

\begin{figure}[h]
\centering
\begin{tikzpicture}

\node at (-1.5,0) {$w_3=$};

\draw(-0.5,-0.5) -- (0,0);
\draw(0.5,-0.5) -- (0,0);
\draw(0,0.71) -- (0,0);
\draw(-0.5,-0.5) -- (0,0.71);
\draw(0.5,-0.5) -- (0,0.71);
\draw(-0.5,-0.5) -- (0.5,-0.5);

\draw[black,fill=black](0,0) circle (0.075cm);
\draw[black,fill=black](-0.5,-0.5) circle (0.075cm);
\draw[black,fill=black](0.5,.-0.5) circle (0.075cm);
\draw[black,fill=black](0,0.71) circle (0.075cm);

\node at (-1.5+4,0) {$w_5=$};

\draw(-0.5+4,-0.5) -- (0+4,0);
\draw(0.5+4,-0.5) -- (0+4,0);
\draw(0+4,0.71) -- (0+4,0);
\draw(-0.71+4,0.25) -- (0+4,0.71);
\draw(4.71,0.25) -- (0+4,0.71);
\draw(-0.5,-0.5) -- (0.5,-0.5);
\draw(-0.71+4,0.25) -- (-0.5+4,-0.5);
\draw(-0.71+4,0.25) -- (4,0);
\draw(4.71,0.25) -- (4,0);
\draw(3.5,-0.5) -- (4.5,-0.5);
\draw(4.71,0.25) -- (4.5,-0.5);

\draw[black,fill=black](0+4,0) circle (0.075cm);
\draw[black,fill=black](-0.5+4,-0.5) circle (0.075cm);
\draw[black,fill=black](0.5+4,.-0.5) circle (0.075cm);
\draw[black,fill=black](0+4,0.71) circle (0.075cm);
\draw[black,fill=black](-0.71+4,0.25) circle (0.075cm);
\draw[black,fill=black](0.71+4,0.25) circle (0.075cm);

\end{tikzpicture}
\caption{The wheel graphs $w_3$ and $w_5$.}\label{figure:wheel graphs}

\end{figure}
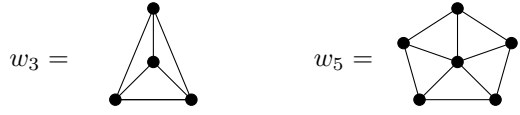

\begin{lemma}
The wheel graph $w_{2k+1}$ satisfies $\#\text{vertices}-\text{maximal valence}=1$ and we may write $s_{2k+1}=w_{2k+1}+R_{2k+1}$, where $R_{2k+1}\in \mf^2H^0(\GC)$.
\end{lemma}

\begin{proof}
The first part of the statement is easily checked. For the second part, let $\gamma$ be any graph representing a class in $H^0(\GC)$ with $\#\text{vertices}=m_\gamma+1$, where $m_\gamma:=\text{maximal valence of }\gamma$. This implies that there must be one vertex which is connected to all other vertices. As $\deg\gamma=-2-\#\text{edges}+2\#\text{vertices}=0$, we have $\#\text{edges}=2m_\gamma$ and since we already fixed a total of $m_\gamma$ edges (emanating from the vertex of maximal valence), we need to decide how to place the other $m_\gamma$ edges. It turns out that in order to respect 1-vertex irreducibility and the condition that all vertices must be of valence at least three, the remaining edges must be placed as to form such a wheel graph.
\end{proof}

Consider the space of brackets of the elements $s_{2k+1}$ in $H^0(\GC)$ corresponding to the conjectural generators $\sigma_{2k+1}$ of $\grt_1$, that is,
\begin{equation*}
\bigoplus\limits_{i, j \geq 1}\mathbb{K}\cdot [s_{2i+1},s_{2j+1}]\subset \mf^2 H^0(\GC)\subset H^0(\GC).
\end{equation*}
Since the filtration on $H^0(\GC)$ is compatible with the Lie algebra structure, the space of brackets is a subspace of $\mf^2 H^0(\GC)$. Moreover, also by compatibility, the bracket descends to a map,
\begin{align*}
\left(\mf^1 H^0(\GC)/ \mf^2 {H^0(\GC)}\right)^{\otimes 2} &\longrightarrow  \mf^2 H^0(\GC)/\mf^3 H^0(\GC)\\
\overline{s}_{2i+1}\otimes \overline{s}_{2j+1} &\longmapsto [\overline{s}_{2i+1}, \overline{s}_{2j+1}]=[\overline{w}_{2i+1},\overline{w}_{2j+1}]=\overline{[w_{2i+1},w_{2j+1}]}.
\end{align*}
In this quotient, the element $[w_{2i+1},w_{2j+1}]$ is represented by the difference of ``bowtie"-graphs as in Figure \ref{figure:hantel}.

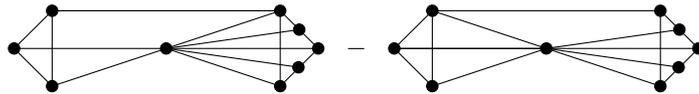
\begin{figure}[ht]
\centering
\begin{tikzpicture}

\draw[black,fill=black](0,0) circle (0.075cm);
\draw[black,fill=black](0.5,0.5) circle (0.075cm);
\draw[black,fill=black](0.5,-0.5) circle (0.075cm);
\draw[black,fill=black](2,0) circle (0.075cm);

\draw[black,fill=black](4,0) circle (0.075cm);
\draw[black,fill=black](3.75,0.25) circle (0.075cm);
\draw[black,fill=black](3.5,0.5) circle (0.075cm);
\draw[black,fill=black](3.74,-0.25) circle (0.075cm);
\draw[black,fill=black](3.5,-0.5) circle (0.075cm);

\draw(0,0) -- (0.5,0.5);
\draw(0,0) -- (0.5,-0.5);
\draw(0,0) -- (2,0);
\draw(0.5,-0.5) -- (0.5,0.5);
\draw(0.5,-0.5) -- (2,0);

\draw(0.5,0.5) -- (3.5,0.5);

\draw(4,0) -- (3.5,0.5);
\draw(4,0) -- (3.5,-0.5);
\draw(3.5,0.5) -- (3.5,-0.5);
\draw(4,0) -- (2,0);
\draw(3.75,0.25) -- (2,0);
\draw(3.5,0.5) -- (2,0);
\draw(3.75,-0.25) -- (2,0);
\draw(3.5,-0.5) -- (2,0);

\node at (4.5,0) {$-$};

\draw[black,fill=black](5,0) circle (0.075cm);
\draw[black,fill=black](5.5,0.5) circle (0.075cm);
\draw[black,fill=black](5.5,-0.5) circle (0.075cm);
\draw[black,fill=black](7,0) circle (0.075cm);

\draw[black,fill=black](9,0) circle (0.075cm);
\draw[black,fill=black](8.75,0.25) circle (0.075cm);
\draw[black,fill=black](8.5,0.5) circle (0.075cm);
\draw[black,fill=black](8.74,-0.25) circle (0.075cm);
\draw[black,fill=black](8.5,-0.5) circle (0.075cm);

\draw(5,0) -- (5.5,0.5);
\draw(5,0) -- (5.5,-0.5);
\draw(5,0) -- (7,0);
\draw(5.5,-0.5) -- (5.5,0.5);
\draw(5.5,0.5) -- (7,0);
\draw(5.5,-0.5) -- (7,0);

\draw(0.5,0.5) -- (3.5,0.5);

\draw(9,0) -- (8.5,0.5);
\draw(9,0) -- (8.5,-0.5);
\draw(8.5,0.5) -- (8.5,-0.5);
\draw(9,0) -- (5,0);
\draw(8.75,0.25) -- (7,0);
\draw(8.5,0.5) -- (5.5,0.5);
\draw(8.75,-0.25) -- (7,0);
\draw(8.5,-0.5) -- (7,0);

\end{tikzpicture}
\caption{Modulo $\mf^3H^0(\GC)$, $[w_3,w_5]$ is given by a nonzero multiple of graphs as depicted above. Individually, we shall refer to such a graphs as ``bowtie"-graphs.}\label{figure:hantel}
\end{figure}

It is obtained by first inserting $w_{2j+1}$ into the highest valent vertex of $w_{2i+1}$ and connecting all but one of the edges (to preserve 1-vertex irreducibility) to the highest valent vertex of $w_{2j+1}$. The second graph is constructed by the same procedure with the roles of $w_{2i+1}$ and $w_{2j+1}$ reversed. One can check that all other graphs produced in $[w_{2i+1},w_{2j+1}]$ lie in $\mf^3H^0(\GC)$. Next consider the composition of the inclusion with the quotient map,
\begin{equation*}
\bigoplus\limits_{i, j \geq 1}\mathbb{K}\cdot [s_{2i+1},s_{2j+1}]\overset{i}{\hookrightarrow} \mf^2H^0(\GC)\overset{\pi}{\twoheadrightarrow} \mf^2 H^0(\GC)/\mf^3 H^0(\GC).
\end{equation*}
It induces an injective map,
\begin{equation*}
\bigoplus\limits_{i, j \geq 1}\mathbb{K}\cdot [s_{2i+1},s_{2j+1}]/\ker(\pi\circ i)\hookrightarrow \mf^2 H^0(\GC)/\mf^3 H^0(\GC).
\end{equation*}
For the next step, we need the following technical tool.

\begin{proposition}[\cite{Willwacher2014}, Section 6.4.1.]
Let 
\begin{align*}
(-)_1:\GC&\rightarrow \graphs(1)\\
\Gamma&\mapsto \Gamma_1,
\end{align*}
where $\Gamma_1$ is obtained by summing over all ways of marking one vertex of $\Gamma$ as external. This map satisfies,
\begin{equation*}
d\Gamma_1-(d\Gamma)_1=(\underset{1}{\circ} \text{---} \underset{2}{\circ}) \circ_2 \Gamma,
\end{equation*}
where the right hand side means that we insert $\Gamma$ in vertex number two and some over all possible ways of reconnecting the edge attached to vertex one to vertices of $\Gamma$.
\end{proposition}

\begin{lemma}\label{lemma:onemap}
The map $(-)_1$ induces a surjective map,
\begin{equation*}
\mf^2H^0(\GC)/\mf^3H^0(\GC)\longrightarrow H^1( \mf^2\icg(1)/\mf^3\icg(1),d_0)=E_1^{2,-1}(1).
\end{equation*}
\end{lemma}

\begin{proof}
Let us first check that the map is well-defined. For this, consider $\Gamma\in \GC$. By Remark \ref{remark:irreducible}, we may assume that it is 1-vertex irreducible. This ensures that $\Gamma_1$ lies in $\icg(1)$. Next, project $\Gamma_1$ onto its two-loop part via the quotient map,
\begin{equation*}
\pi:\icg(1)\rightarrow \mf^2\icg(1)/\mf^3\icg(1).
\end{equation*}
Note that there are no two-loop graphs of weight one in $\icg(1)$ (i.e. with only one edge adjacent to the unique external vertex). Thus, the composition $\pi\circ (-)_1$ is a chain map, that is,
\begin{equation*}
d\pi\Gamma_1=\pi (d\Gamma)_1,
\end{equation*}
and it induces a map on the level of cohomology,
\begin{equation}\label{onemapcohom}
H^0(\GC)\rightarrow H^1( \mf^2\icg(1)/\mf^3\icg(1),d_0)=E_1^{2,-1}(1).
\end{equation}
Let now $\Gamma\in \mf^2H^0(\GC)$ be such that it represents a non-zero equivalence class in the quotient 
\begin{equation*}
\mf^2 H^0(\GC)/\mf^3 H^0(\GC).
\end{equation*}
If $m$ denotes the vertex of maximal valence of $\Gamma$ and $N$ the number of vertices of $\Gamma$, we have $N-\text{val}(m)\geq 2$. In particular, this means that there is at least one vertex which is not directly connected to $m$. Moreover, the part $\Gamma_2$ of $\Gamma$ which represents the non-zero element in the quotient has exactly one such vertex. Call it $w$, and assume the graph we are considering has $k$ edges. Using the elementary formula,
\begin{equation*}
\sum\limits_{v \text{ vertex of }\Gamma_2}\val(v)=2k,
\end{equation*}
and the fact that $k=2N-2$ (since $\Gamma$ is of degree zero in $\GC$), and that $\text{val}(m)=N-2$, we find,
\begin{equation}\label{eq:valence}
\sum\limits_{v \text{ vertex of } \Gamma_2}\val(v)=\val(m)+\val(w)+\sum\limits_{v\neq w,m}\text{val}(v)=2k=4N-4.
\end{equation}
As all vertices are at least trivalent,
\begin{equation*}
\sum\limits_{v\neq w,m}\text{val}(v)\geq 3(N-2).
\end{equation*}
If there is equality, that is all vertices other than $m$ and $w$ are trivalent, we find $\text{val}(w)=4$. If the inequality is strict, first of all,
\begin{equation*}
\val(m)+\val(w)+3(N-2)<4N-4,
\end{equation*}
and thus,
\begin{equation*}
3\leq \text{val}(w)<N+2-m=4,
\end{equation*}
i.e. $\text{val}(w)=3$. Equation \eqref{eq:valence} then reads,
\begin{equation*}
\val(m)+3+\sum\limits_{v\neq w,m}\text{val}(v)=4N-4,
\end{equation*}
which implies that there is exactly one 4-valent vertex, and with the exception of $m$, the rest is trivalent. From this we deduce, that the only graphs which represent non-zero elements in the quotient above are of the following form. Either the non-maximal valent vertices make up two loops which do not share an edge and which are only connected through one edge, and all except one non-maximal valent vertices are connected to $m$ (i.e. a ``bowtie"-graph), or the non-maximal valent vertices make up two loops which share at least one edge and again all but one non-maximal valent vertices are connected to $m$. Examples of the two cases are depicted in Figure \ref{figure:nonzero}. Applying $(-)_1$ to $\Gamma$, we find that the only part which produces a two-loop graph in $\icg(1)$ is given by the part which is non-zero in the quotient $\mf^2H^0(\GC)/\mf^3H^0(\GC)$, and only when we mark $m$ as external. The map from \eqref{onemapcohom} therefore factors through the projection,
\begin{equation*}
H^0(\GC)\rightarrow \mf^2 H^0(\GC)/\mf^3 H^0(\GC),
\end{equation*}
and thus induces a map
\begin{equation*}
\mf^2H^0(\GC)/\mf^3H^0(\GC)\longrightarrow H^1( \mf^2\icg(1)/\mf^3\icg(1),d_0)=E_1^{2,-1}(1).
\end{equation*}
It is surjective, since (for $i,j\geq 1$) all the $\theta_{2i,2j}$-graphs which generate $E_1^{2,-1}(1)$ lie in the image. They are obtained by marking the maximal valent vertex of a graph similar to the one on the right in Figure \ref{figure:nonzero} as external.

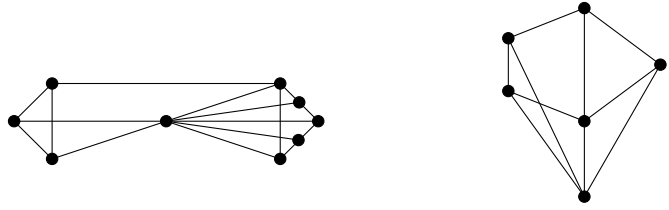
\begin{figure}[ht]
\centering
\begin{tikzpicture}

\draw[black,fill=black](0,0) circle (0.075cm);
\draw[black,fill=black](0.5,0.5) circle (0.075cm);
\draw[black,fill=black](0.5,-0.5) circle (0.075cm);
\draw[black,fill=black](2,0) circle (0.075cm);

\draw[black,fill=black](4,0) circle (0.075cm);
\draw[black,fill=black](3.75,0.25) circle (0.075cm);
\draw[black,fill=black](3.5,0.5) circle (0.075cm);
\draw[black,fill=black](3.74,-0.25) circle (0.075cm);
\draw[black,fill=black](3.5,-0.5) circle (0.075cm);

\draw(0,0) -- (0.5,0.5);
\draw(0,0) -- (0.5,-0.5);
\draw(0,0) -- (2,0);
\draw(0.5,-0.5) -- (0.5,0.5);
\draw(0.5,-0.5) -- (2,0);

\draw(0.5,0.5) -- (3.5,0.5);

\draw(4,0) -- (3.5,0.5);
\draw(4,0) -- (3.5,-0.5);
\draw(3.5,0.5) -- (3.5,-0.5);
\draw(4,0) -- (2,0);
\draw(3.75,0.25) -- (2,0);
\draw(3.5,0.5) -- (2,0);
\draw(3.75,-0.25) -- (2,0);
\draw(3.5,-0.5) -- (2,0);

\draw[black,fill=black](6.5,1.4-1) circle (0.075cm);
\draw[black,fill=black](8.5,1.75-1) circle (0.075cm);
\draw[black,fill=black](6.5,2.1-1) circle (0.075cm);
\draw[black,fill=black](7.5,2.5-1) circle (0.075cm);
\draw[black,fill=black](7.5,1-1) circle (0.075cm);
\draw[black,fill=black](7.5,0-1) circle (0.075cm);

\draw(7.5,0-1) -- (8.5,1.75-1);
\draw(7.5,0-1) -- (7.5,2.5-1);
\draw(7.5,0-1) -- (6.5,1.4-1);
\draw(7.5,0-1) -- (6.5,2.1-1);
\draw(6.5,1.4-1) -- (6.5,2.1-1);
\draw(6.5,2.1-1) -- (7.5,2.5-1);
\draw(7.5,2.5-1) -- (8.5,1.75-1);
\draw(8.5,1.75-1) -- (7.5,1-1);
\draw(6.5,1.4-1) -- (7.5,1-1);

\end{tikzpicture}
\caption{Graphs which represent a non-zero class in the quotient $\mf^2H^0(\GC)/\mf^3H^0(\GC)$.}\label{figure:nonzero}
\end{figure}

\end{proof}

\begin{proposition}\label{prop:itsaniso}
The map from Lemma \ref{lemma:onemap} restricted to the quotient space of brackets modulo $\ker(\pi\circ i)$ is an isomorphism. It can be normalized to map generators to generators, i.e.
\begin{align*}
\Omega:\bigoplus\limits_{i, j \geq 1}\mathbb{K}\cdot [s_{2i+1},s_{2j+1}]/\ker(\pi\circ i)&\rightarrow E_1^{2,-1}(1)\\
\overline{[s_{2i+1},s_{2j+1}]}&\mapsto \theta_{2i,2j}
\end{align*}
\end{proposition}

\begin{proof}
We have already seen that in the quotient $\mf^2H^0(\GC)/\mf^3H^0(\GC)$ brackets $[s_{2i+1},s_{2j+1}]$ are represented by the differences of ``bowtie"-graphs. Applying the map from Lemma \ref{lemma:onemap} to this class, we obtain a difference of two-loop graphs in $\icg(1)$ (see Figure \ref{figure:hantelicg}). Denote it by $D_{2i+1,2j+1}$. It is cohomologous to a multiple of the theta graph $\theta_{2i,2j}$ (see Figure \ref{figure:thetaicg}). 
\begin{figure}[ht]
\centering
\begin{tikzpicture}

\draw[black,fill=black](0,0) circle (0.05cm);
\draw[black,fill=black](0.5,0.5) circle (0.05cm);
\draw[black,fill=black](0.5,-0.5) circle (0.05cm);

\draw(2,0) circle (0.1cm);
\node at (2,-0.5) {$1$};

\draw[black,fill=black](4,0) circle (0.05cm);
\draw[black,fill=black](3.75,0.25) circle (0.05cm);
\draw[black,fill=black](3.5,0.5) circle (0.05cm);
\draw[black,fill=black](3.74,-0.25) circle (0.05cm);
\draw[black,fill=black](3.5,-0.5) circle (0.05cm);

\draw(0,0) -- (0.5,0.5);
\draw(0,0) -- (0.5,-0.5);
\draw(0,0) -- (2,0);
\draw(0.5,-0.5) -- (0.5,0.5);
\draw(0.5,-0.5) -- (2,0);

\draw(0.5,0.5) -- (3.5,0.5);

\draw(4,0) -- (3.5,0.5);
\draw(4,0) -- (3.5,-0.5);
\draw(3.5,0.5) -- (3.5,-0.5);
\draw(4,0) -- (2,0);
\draw(3.75,0.25) -- (2,0);
\draw(3.5,0.5) -- (2,0);
\draw(3.75,-0.25) -- (2,0);
\draw(3.5,-0.5) -- (2,0);

\node at (4.5,0) {$-$};

\draw[black,fill=black](5,0) circle (0.05cm);
\draw[black,fill=black](5.5,0.5) circle (0.05cm);
\draw[black,fill=black](5.5,-0.5) circle (0.05cm);

\draw(7,0) circle (0.1cm);

\node at (7,-0.5) {$1$};

\draw[black,fill=black](9,0) circle (0.05cm);
\draw[black,fill=black](8.75,0.25) circle (0.05cm);
\draw[black,fill=black](8.5,0.5) circle (0.05cm);
\draw[black,fill=black](8.74,-0.25) circle (0.05cm);
\draw[black,fill=black](8.5,-0.5) circle (0.05cm);

\draw(5,0) -- (5.5,0.5);
\draw(5,0) -- (5.5,-0.5);
\draw(5,0) -- (7,0);
\draw(5.5,-0.5) -- (5.5,0.5);
\draw(5.5,0.5) -- (7,0);
\draw(5.5,-0.5) -- (7,0);

\draw(0.5,0.5) -- (3.5,0.5);

\draw(9,0) -- (8.5,0.5);
\draw(9,0) -- (8.5,-0.5);
\draw(8.5,0.5) -- (8.5,-0.5);
\draw(9,0) -- (5,0);
\draw(8.75,0.25) -- (7,0);
\draw(8.5,0.5) -- (5.5,0.5);
\draw(8.75,-0.25) -- (7,0);
\draw(8.5,-0.5) -- (7,0);

\end{tikzpicture}
\caption{The vertex with highest valence is marked as external to obtain an element $D_{3,5}\in \icg(1)$.}\label{figure:hantelicg}
\end{figure}
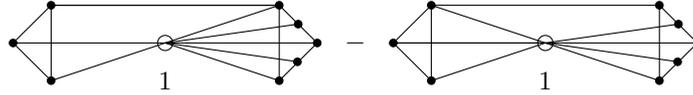

To see this, denote by $E_{2i,2j}\in \icg(1)$ the ``figure-8" two-loop graph of degree zero as depicted in Figure \ref{figure:E}. The differential $d_0$ splits the vertex of valence five in 10 different ways. The four graphs which are still of ``figure-8" type graphs are zero due to symmetry. By ordering the edges in a consistent way, we are left with,
\begin{equation*}
d_0E_{2i,2j}=D_{2i+1,2j+1}+4\cdot \theta_{2i,2j}.
\end{equation*}
Since the theta-graphs $\theta_{2i,2j}$ generate $E_1^{2,-1}(1)$, $\Omega$ is clearly surjective, and we can normalize it such that it satisfies $\Omega(\overline{[s_{2i+1},s_{2j+1}]})=\theta_{2i,2j}$.
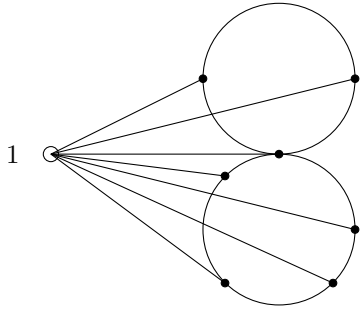
\begin{figure}[ht]
\centering
\begin{tikzpicture}

\draw (0,1) circle (1cm);
\draw (0,-1) circle (1cm);

\draw[black,fill=black](1,1) circle (0.05cm);
\draw[black,fill=black](-1,1) circle (0.05cm);

\draw[black,fill=black](0.71,-1.71) circle (0.05cm);
\draw[black,fill=black](-0.71,-1.71) circle (0.05cm);
\draw[black,fill=black](1,-1) circle (0.05cm);
\draw[black,fill=black](-0.71,-0.29) circle (0.05cm);

\draw[black,fill=black](0,0) circle (0.05cm);

\draw(-3,0) circle (0.1cm);

\draw(-3,0) -- (0,0);

\draw(1,1) -- (-3,0);
\draw(-1,1) -- (-3,0);

\draw(-0.71,-0.29) -- (-3,0);
\draw(1,-1) -- (-3,0);
\draw(0.71,-1.71) -- (-3,0);
\draw(-0.71,-1.71) -- (-3,0);

\node at (-3.5,0) {$1$};

\end{tikzpicture}
\caption{The graph $E_{2,4}$ which is mapped $D_{3,5}+4\cdot \theta_{2,4}$ by $d_0$.}\label{figure:E}
\end{figure}

For the injectivity, assume that for $k\geq 8$, $A:=\sum\limits_{i=1}^{\lfloor\frac{k-4}{2}\rfloor} a_i [s_{2i+1},s_{k-1-2i}]\in \ker(\Omega)$, i.e.
\begin{equation*}
\Omega(A)=\sum\limits_{i=1}^{\lfloor\frac{k-4}{2}\rfloor} a_{i} D_{2i+1,k-1-2i}=0\in E_1^{2,-1}(1).
\end{equation*}
This implies the existence of some $B \in \mf^2\icg(1)/\mf^3\icg(1)$ of degree $0$ such that $\Omega(A)=d_0(B)$. Note that $\Omega(A)$ and $B$ are both of weight $k$, and by degree reasons of all vertices (internal and external) the external vertex is the one of highest valence. Next, consider the graphs in $\GC$ obtained by, first, marking the unique external vertex of $\Omega(A)$ and of $B$ as internal again, and then summing over all ways of labeling the vertices to make them indistinguishable. Call them $A'$ and $B'$, respectively. Note that $A'$ is a scalar multiple of $A$. Apply the differential of $\GC$ on $B'$. We remark that the part (denoted by $Q_1$ in the following) of the differential acting on the vertex that was previously external will lie in $\mf^3 H^0(\GC)$, since it will invariably produce at least two vertices which are not directly connected to the vertex of maximal valence. The remaining part (denoted by $Q_2$) of the differential will only operate on vertices which were previously internal. It mimics the differential $d_0$ on $\icg(1)$, and thus produces a multiple of $A'$, hence a multiple of $A$. To summarize, the reasoning above implies that there is a $\lambda \in \mathbb{K}^\times$ such that $dB'=\lambda A+Q_1$, and thus,
\begin{equation*}
A=\lambda^{-1}dB' \in \mf^2H^0(\GC)/\mf^3H^0(\GC),
\end{equation*}
Equivalently, $A=0\in \mf^2H^0(\GC)/\mf^3H^0(\GC)$, and $\Omega$ is injective.
\end{proof}

\begin{corollary}
The map $(-)_1$ from Lemma \ref{lemma:onemap} is an isomorphism, and we may thus identify the following spaces,
\begin{equation*}
\bigoplus\limits_{i, j \geq 1}\mathbb{K}\cdot [s_{2i+1},s_{2j+1}]/\ker(\pi\circ i)\cong\mf^2H^0(\GC)/\mf^3H^0(\GC)\cong E_1^{2,-1}(1).
\end{equation*}
\end{corollary}

\begin{proof}
The map $\Omega$ describes the composition,
\begin{equation*}
\bigoplus\limits_{i, j \geq 1}\mathbb{K}\cdot [s_{2i+1},s_{2j+1}]/\ker(\pi\circ i)\hookrightarrow\mf^2H^0(\GC)/\mf^3H^0(\GC) \twoheadrightarrow E_1^{2,-1}(1).
\end{equation*}
Since it is an isomorphism, the statement easily follows.
\end{proof}

\begin{proof}[Proof of Theorem \ref{thm:compatible}]
Clearly, the foregoing discussion implies,
\begin{align*}
\mf^2H^0(\GC)/\mf^3H^0(\GC)&\overset{\Omega}{\cong} E_1^{2,-1}(1)\overset{\Phi^{-1}}{\cong} \mf^2\grt_1/\mf^3\grt_1\\
\overline{[s_{2i+1},s_{2j+1}]}&\mapsto \text{ }\text{ }\theta_{2i,2j}\text{ }\text{ }\mapsto \{\sigma_{2i+1},\sigma_{2j+1}\}\mod \mf^3\grt_1,
\end{align*}
from which it follows that $\Phi^{-1}\circ\Omega=\omega$.
\end{proof}

\bibliographystyle{plain}
\bibliography{bibliography}

\end{document}